\documentclass[reqno]{amsart}
\usepackage{amsmath,amsthm,amsfonts, amssymb,amscd,enumitem}
\usepackage[all]{xy}
\numberwithin{equation}{section}
  \newtheorem{thm}{Theorem}[section]
  \newtheorem{lem}[thm]{Lemma}
  \newtheorem{prop}[thm]{Proposition}
  \newtheorem{cor}[thm]{Corollary}
  \theoremstyle{definition}
  \newtheorem{defn}[thm]{Definition}
  \newtheorem{exm}[thm]{Example}
  \newtheorem{rmk}[thm]{Remark}
  
 \newcommand\ra{\rightarrow}

\newcommand{\lex}{\,\overrightarrow{\times}\,}

 \newcommand\mI{\mathcal{I}}

 \newcommand\s{\subseteq}
 
 \newcommand\supp{\mathrm{Supp}}

 \newcommand\B{\mathrm{B}}
 
  \newcommand\lam{\lambda}

\newcommand{\id}{\mbox{\rm Id}}
\newcommand\spec{\mathrm{Spec}}
 \numberwithin{equation}{section}
\def\iff{if and only if }
\def\im{ \mathrm{Im} }

\let\Right\right
\let\Left\left
\makeatletter
\def\right#1{\Right#1\@ifnextchar){\!\right}{}}
\def\left#1{\Left#1\@ifnextchar({\!\left}{}}

\begin{document}
\title[On EMV-algebras with square roots]{On EMV-algebras with square roots}
\author[Anatolij Dvure\v{c}enskij and Omid Zahiri]{Anatolij Dvure\v{c}enskij$^{^{1,2,3}}$, Omid Zahiri$^{^{1,*}}$}

\date{}%
\thanks{The paper acknowledges the support by the grant of
the Slovak Research and Development Agency under contract APVV-20-0069  and the grant VEGA No. 2/0142/20 SAV,  A.D}
\thanks{The project was also funded by the European Union's Horizon 2020 Research and Innovation Programme on the basis of the Grant Agreement under the Marie Sk\l odowska-Curie funding scheme No. 945478 - SASPRO 2, project 1048/01/01,  O.Z}
\thanks{* Corresponding Author: Omid Zahiri}
\address{$^1$Mathematical Institute, Slovak Academy of Sciences, \v{S}tef\'anikova 49, SK-814 73 Bratislava, Slovakia}
\address{$^2$Palack\' y University Olomouc, Faculty of Sciences, t\v r. 17. listopadu 12, CZ-771 46 Olomouc, Czech Republic}
\address{$^3$ Depart. Math., Constantine the Philosopher University in Nitra, Tr. A. Hlinku 1, SK-949 01 Nitra, Slovakia}

\email{dvurecen@mat.savba.sk, zahiri@protonmail.com}
\thanks{}



\begin{abstract}
A square root is a unary operation with some special properties. In the paper, we introduce and study square roots on EMV-algebras.
First, the known properties of square roots defined on MV-algebras will be generalized for EMV-algebras, and we also find some new ones for MV-algebras. We use square roots to characterize EMV-algebras. Then, we find a relation between the square root of an EMV-algebra and the square root on its representing EMV-algebra with top element.
We show that each strict EMV-algebra has a top element and we investigate the relation between divisible EMV-algebras and EMV-algebras with a special square root. Finally, we present square roots on tribes, EMV-tribes, and we present a complete characterization of any square root on an MV-algebra and on an EMV-algebra by group addition in the corresponding unital $\ell$-group.
\end{abstract}

\date{}

\maketitle

\section{ Introduction }

MV-algebras were introduced by C.C. Chang in \cite{Cha1,Cha2} as basic algebraic tools for many-valued evaluation to provide an
algebraic proof of the completeness of the \L ukasiewicz infinite-valued proposition calculus.
Another fundamental study on MV-algebras was done in \cite{Mun}, which provides a one-to-one characterization of MV-algebras as intervals in
unital Abelian $\ell$-groups. Nowadays, MV-algebras are applied in many areas of mathematics, logic, computing, etc. with very deep results. Therefore, this structure has also many generalizations, like BL-algebras, hoops, BCK-algebras, quasi MV-algebras, etc. In \cite{georgescu}, a non-commutative generalization of MV-algebras, pseudo MV-algebras, was introduced. These algebras are also known as generalized MV-algebras, see \cite{Rac}. Nowadays, there are studied also non-commutative versions of BL-algebras or hoops.

Recently, we introduced EMV-algebras \cite{Dvz} as a common generalization of MV-algebras and generalized Boolean algebras.
A top element for EMV-algebras is not necessarily assumed. In a special case, if it admits a top element, then it is equivalent to an MV-algebra. Conjunction, disjunction, and $\oplus$ are defined but a complement is defined only locally, i.e. if $a$ is a Boolean element, then $[0,a]$ is an MV-algebra and a complement of any element in $[0,a]$ is defined. Each EMV-algebra can be covered by $\{[0,a]\colon a$ is a Boolean element$\}$.
We have to note that the class of EMV-algebras is not a variety because it is not closed under forming subalgebras with respect to the original operations. Some results known for MV-algebras were extended for EMV-algebras and some new results were obtained for MV-algebras.

In \cite{Dvz}, it was proved a basic representation result that each EMV-algebra $M$ either has a top element or there exists an EMV-algebra $N$ with top element such that $M$ can be embedded into $N$ as a maximal ideal of $N$. States as analogs of finitely additive measures were investigated in \cite{297} and morphisms and free EMV-algebras were described in \cite{Dvz3}. The Loomis--Sikorski theorem for these algebras was established in \cite{lomi}. We showed also that every EMV-algebra $M$ is a homomorphic image of $\B(M)$, the generalized Boolean algebra $R$-generated by $M$, where a homomorphism is a homomorphism of generalized effect algebras (see \cite{Dvz4}).
Since the class of all EMV-algebras, $\mathsf{EMV}$, is not a variety, in \cite{Dvz30}, we found the least variety containing $\mathsf{EMV}$.

Square roots have been studied in many algebraic structures.
In the case of non-idempotent semigroup operations, square roots can carry an important part of information.
For example, further binary operations can be easily built from square roots and the given semigroup operations. If the underlying lattice is given by the real unit interval, then the arithmetic, respectively geometric, mean arises in this way.
U. H{\"o}hle in \cite{Hol} studied square roots on the class of residuated lattices and its subclasses such as MV-algebras.
He characterized MV-algebras with square roots in a general case. He proved every MV-algebra with square root belongs to only one of three classes of MV-algebras. He also introduced and investigated strict MV-algebras and found relations between strict MV-algebras, injective MV-algebras, and divisible MV-algebras in the class of complete MV-algebras. It was proved that every complete MV-algebra with square roots is  either a complete Boolean algebra or a Boolean valued model of the real unit interval viewed as an MV-algebra or a product of both. R. Ambrosio in \cite{Amb} continued to investigate strict MV-algebras. She classified strict MV-algebras by using
the concept of  2-atomless. She showed that each strict MV-algebra contains a cyclic element of order $2^{n+1}$ for each integer $n\geq 0$. Some new results on square roots on GL-monoids can be found in \cite{Hol1}, and also in \cite{NPM}, there are interesting ones on square roots. In \cite{BeR}, it was proved that the class of residuated lattices with square roots is a variety. It is clear that every BL-algebra corresponding to a continuous $t$-norm has square roots.

We note that the notion of the square root is also used in \cite{GLP} for quasi MV-algebras and in \cite{ChDu} for quasi-pseudo MV-algebras but in a completely different sense as in our contribution.

Our main goal in this article is to introduce and study square roots of EMV-algebras:
\vspace{1mm}
	\begin{itemize}[nolistsep]
\item[(1)] Find a relation between square roots on an EMV-algebra and its representing EMV-algebra with top element.

\item[(2)] Characterize strict EMV-algebras.

\item[(3)] Classify EMV-algebras with square roots.

\item[(4)] Study divisible EMV-algebras and characterize them by square roots.
\item[(5)]  Give a complete characterization of any square root on any MV-algebra and on any EMV-algebra.
\end{itemize}

The paper is organized as follows:
Section 2 contains basic notions and facts about MV-algebras, square roots, and EMV-algebras which will be used in the next sections.
In Section 3, known properties of square roots defined on MV-algebras are generalized for EMV-algebras, and we also find some new properties. We use square roots to characterize EMV-algebras and we find a relation between the square root of an EMV-algebra and the square root on its representing EMV-algebra with top element.
In Section 4, we introduce strict EMV-algebras, and we show that each strict EMV-algebra has a top element. Moreover, we present a classification of EMV-algebras with square roots. In Section 5,
relations between strict EMV-algebras with some other subclasses of EMV-algebras such as divisible and locally complete EMV-algebras are investigated. Moreover, we present square roots on tribes, EMV-tribes, and a complete characterization of any square root on every MV-algebra and every EMV-algebra by group addition in the corresponding unital $\ell$-group.

\section{Preliminaries}

In this section, we gather some preliminary results about generalized Boolean algebras and EMV-algebras, which will be needed in the following sections. Some of the results of this section (Lemma \ref{lemsqMV}, Proposition \ref{prop3.7}, and Lemma \ref{lem3.2}) are new, and they are related to MV-algebras, so we state them in this section with short proofs.

Recall that a generalized Boolean algebra is a relatively complemented distributive lattice with a bottom element. By the proof of \cite[Thm 2.2]{CoDa}, each generalized Boolean algebra is either a Boolean algebra or can be embedded into the Boolean algebra of subsets of $\mathrm{MaxI}(B)$ as its maximal ideal.

We note that an {\it MV-algebra} is an algebra $(M;\oplus,',0,1)$ of type $(2,1,0,0)$, where $(M;\oplus,0)$ is a commutative monoid with the neutral element $0$ and for all $x,y\in M$, we have:
\vspace{1mm}
\begin{enumerate}[nolistsep]
	\item[(i)] $x''=x$;
	\item[(ii)] $x\oplus 1=1$;
	\item[(iii)] $x\oplus (x\oplus y')'=y\oplus (y\oplus x')'$;
 \item[(iv)] $0'=1$.
\end{enumerate}
\noindent
In any MV-algebra $(M;\oplus,',0,1)$,
we can also define the following  operations:
\[
x\odot y:=(x'\oplus y')',\quad x\ominus y:=(x'\oplus y)'.
\]
MV-algebras are intimately connected with Abelian unital $\ell$-groups: Let $G$ be an Abelian lattice-ordered group ($\ell$-group) written additively, an element $u\ge 0$ of $G$ is said to be a {\it strong unit} if, given $g\in G$, there is an integer $n\ge 1$ such that $g\le nu$. The couple $(G,u)$, where $G$ is an $\ell$-group and $u$ is a fixed strong unit of $G$, is said to be a {\it unital $\ell$-group}. On the interval $[0,u]$, we define $x\oplus y=(x+y)\wedge u$ and $x'=u-x$, $x,y \in M$. Then $\Gamma(G,u)=([0,u];\oplus,',0,u)$ is an MV-algebra. Due to the famous result by Mundici, \cite{Mun}, every MV-algebra is isomorphic to  $\Gamma(G,u)$ for some unital $\ell$-group $(G,u)$. Moreover, there is a categorical equivalence between the category of MV-algebras and the category of unital $\ell$-groups. For more information on MV-algebras and their relationship with unital $\ell$-groups, we recommend consulting with \cite{CDM}.

For any integer $n\in\mathbb N$ and any $x\in M$, we can define $0.x =0$, and $n.x=(n-1).x\oplus x$, $n\ge 1$.
An MV-algebra $(M;\oplus,',0,1)$ is called {\it strongly atomless} if, for each $x\in M\setminus \{0\}$, there exists $z\in M$ such that
$0<z<x$ and $x\odot z'\leq z$ (see \cite{Bel,Amb}).

\begin{defn}\cite{Hol}
Let $M$ be an MV-algebra. A mapping $s:M\to M$ is called a {\em square root} if it satisfies the following conditions:
\begin{enumerate}[nolistsep]
\item[(i)] for all $x\in M$, $s(x)\odot s(x)=x$;
\item[(ii)] for all $x,y\in M$, $y\odot y\leq x$ implies that $y\leq s(x)$.
\end{enumerate}

We note that sometimes ones write $s(x)=x^{\frac{1}{2}}$, $x\in M$, see \cite{Hol,BeR}. We note that due to \cite[Cor 3]{BeR}, the class of MV-algebras with square roots is a variety.

An MV-algebra with a square root $s$ is called {\it strict} \iff $s(0)=s(0)'$. We say also that the square root $s$ is strict.
\end{defn}

\begin{lem}\label{lemsqMV}
Let $M$ be an MV-algebra with a square root $s$. Then for each $x\in M$:
\begin{enumerate}[nolistsep]
\item[{\rm (i)}] $s(x')'\oplus s(x')'=x$, consequently
$s(x')'\leq x\leq s(x)$.
\item[{\rm (ii)}] If $x=s(x')'$ or $x=s(x)$, then $x$ is a Boolean element.
\item[{\rm (ii)}] If $x$ is not a Boolean element, then $s(x')'< x< s(x)$.
\end{enumerate}
\end{lem}

\begin{proof}
(i) By definition for each $x\in M$ we have $s(x')\odot s(x')=x'$ which implies that
$s(x')'\oplus s(x')'=x$.

(ii) If $x=s(x)$, then by definition $x\odot x=s(x)\odot s(x)=x$. If $x=s(x')'$, then $x'=s(x')$ which implies that
$x'\odot x'=s(x')\odot s(x')=x'$. So, by \cite[Thm 1.5.3]{CDM}, $x$ is a Boolean element of $M$.

(iii) It follows from (i) and (ii).
\end{proof}

\begin{prop}\label{prop3.7}
Let $s$ be a square root on an MV-algebra $(M;\oplus,',0,1)$. For each $x\in M$, we have
$(s(x')\odot s(x'))'=s(x)\odot s(x)$.
\end{prop}

\begin{proof}
Let $x\in M$. From
$s(x')\odot s(x')=x'$ and $s(x)\odot s(x)=x$ we get that $s(x)\odot s(x)=(s(x')\odot s(x'))'=s(x')'\oplus s(x')'$.
Also, $x\leq s(x)$ and $x'\leq s(x')$, so $s(x)'\leq x'\leq s(x')$.
\end{proof}

Recently, in \cite{Dvz} an extension of generalized Boolean algebras and MV-algebras, called EMV-algebras, was introduced.

An algebra $(M;\vee,\wedge,\oplus,0)$ of type $(2,2,2,0)$ is called an
{\it extended MV-algebra}, an {\it EMV-algebra} in short, if it satisfies the following conditions:
	\vspace{1mm}
	\begin{itemize}[nolistsep]
	\item[{\rm (E1)}] $(M;\vee,\wedge,0)$ is a distributive lattice with the least element $0$;
	\item[{\rm (E2)}]  $(M;\oplus,0)$ is a commutative ordered monoid with the neutral element $0$;
	\item[{\rm (E3)}] for each $a\in \mI(M):=\{x\in M\mid x\oplus x=x\} $, the element
	$\lambda_{a}(x)=\min\{z\in[0,a]\mid x\oplus z=a\}$
	exists in $M$ for all $x\in [0,a]$, and the algebra $([0,a];\oplus,\lambda_{a},0,b)$ is an MV-algebra;
	\item[{\rm (E4)}] $\mI(M)$ is a full subset of $M$, that is for each $x\in M$, there is $a\in \mI(M)$ such that
	$x\leq a$.
\end{itemize}
The elements of $\mathcal I(M)$ are said to be {\it idempotents}.

An EMV-algebra $M$ is called {\em proper} if it does not have a top element. Clearly, if in a generalized Boolean algebra we put $\oplus =\vee$, for each $x\in [0,a]$, $\lambda_a(x)$ is a relative complement of $x$ in $[0,a]$, then every generalized Boolean algebra can be viewed as an EMV-algebra where top element is not necessarily assumed.
If $1$ is a top element of an EMV-algebra $M$, by (E3), $([0,1];\oplus,\lambda_1,0,1)=(M;\oplus,',0,1)$ is an MV-algebra. Conversely, if $(M;\oplus,',0,1)$ is an MV-algebra, then $(M;\vee,\wedge,\oplus,0)$ is an EMV-algebra with top element $1$. In addition, every EMV-algebra $(M;\vee,\wedge,\oplus,0)$ with top element $1$ is termwise equivalent to an MV-algebra $(M;\oplus,',0,1)$.

Let $(M;\vee,\wedge,\oplus,0)$ be an EMV-algebra.
The lattice structure of $M$  yields a partial order relation on $M$, denoted by $\leq$, that is $x\leq y$ iff $x\vee y=y$ iff $x\wedge y=x$. Also, if $a$ is a fixed idempotent element of $M$, there is a
partial order relation $\preccurlyeq_a$  on
the MV-algebra $([0,a];\oplus,\lambda_a,0,a)$ defined by $x\preccurlyeq_a y$ iff $\lambda_a(x)\oplus y=a$.
By \cite{Dvz3}, we know that for each $x,y\in [0,a]$, we have $x\leq y\Leftrightarrow x\preccurlyeq_a y$.
Also, if  $x,y \le b\in \mathcal I(M)$, then
$x\preccurlyeq_a y \Leftrightarrow  x\leq y\Leftrightarrow x\preccurlyeq_b y$.

\begin{prop}\label{2.2}{\rm \cite[Prop 3.9]{Dvz}}
	Let $(M;\vee,\wedge,\oplus,0)$ be an EMV-algebra and $a,b\in \mI(M)$ such that $a\leq b$.
	Then, for each $x\in [0,a]$, we have
	\vspace{1mm}
	\begin{itemize}[nolistsep]
		\item[{\rm (i)}]  $\lam_a(x)=\lam_b(x)\wedge a$.
		\item[{\rm (ii)}]   $\lam_b(x)=\lam_a(x)\oplus \lam_b(a)$.
		\item[{\rm (iii)}] $\lam_b(a)$ is an idempotent, and $\lam_a(a)=0$.
	\end{itemize}
\end{prop}

\begin{lem}\label{2.3}{\rm \cite[Lem 5.1]{Dvz}}
Let $(M;\vee, \wedge,\oplus,0)$ be an EMV-algebra. For all $x,y\in M$, we define
\[ x\odot y=\lam_a(\lam_a(x)\oplus \lam_a(y)), \]
where $a\in\mI(M)$ and $x,y\leq a$. Then $\odot:M\times M\ra M$ is an order preserving, associative, and well-defined binary operation on $M$ which does not depend on $a\in \mI(M)$ with $x,y \le a$.
In addition, if $x,y \in M$, $x\le y$, then $y \odot \lambda_a(x)=y\odot \lambda_b(x)$
for all idempotents $a,b$ of $M$ with $x,y\le a,b$. So, we denote $x\ominus y$ by $x \odot \lambda_a(y)$ for all $a\geq y,x$.
\end{lem}

On every EMV-algebra, we can define a partial operation $+$ as follows: The element $x+y$ is defined in an EMV-algebra $M$ iff $x,y \in M$ are such that $x\odot y=0$, then we set $x+y:= x\oplus y$. It is possible to show that $+$ is commutative, associative, and cancellative. If $M=\Gamma(G,u)$, then $x+y$ is in fact the group addition of $x$ and $y$ in $G$. If $a,b\in \mathcal I(M)$, $a\wedge b=0$, then $a+b=a\vee b$, and $x\wedge (a+b)= (x\wedge a)+(x\wedge b)= (x\wedge a)\vee (x\wedge b)$.

For any integer $n\ge 1$ and any $x$ of an EMV-algebra $M$, we can define $x^1 =x$, $x^n=x^{n-1}\odot x$, $n\ge 2$, and $x^0=1$ if $1$ is defined in $M$.

Let $(M_1;\vee,\wedge,\oplus,0)$ and $(M_2;\vee,\wedge,\oplus,0)$ be EMV-algebras. A map $f:M_1\ra M_2$ is called an EMV-{\it homomorphism}
if $f$ preserves the operations $\vee$, $\wedge$, $\oplus$ and $0$, and for each $b\in\mI(M_1)$ and for each $x\in [0,b]$, $f(\lam_b(x))= \lam_{f(b)}(f(x))$. An EMV-homomorphism $f:M_1\to M_2$ is said to be {\em strong} if, for each $b\in \mI(M_2)$, there exists $a\in\mI(M_1)$ such that $b\leq f(a)$. Easy calculations show that each EMV-morphism preserves the operation $\odot$, too (see \cite{Dvz3,Dvzp}).

The following important result on representing EMV-algebras was established in \cite[Thm 5.21]{Dvz}.

\begin{thm}\label{2.4}{\rm [Basic Representation Theorem]}
Every EMV-algebra $M$ either has a top element or $M$ can be embedded into an EMV-algebra $N$  with top element as a maximal ideal of $N$ such that every element $x\in N$ is either the image of some element from $M$ or $x$ is the complement of the image of some element from $M$.
\end{thm}

The EMV-algebra $N$ with top element in the latter theorem is unique up to isomorphism and it is said to be {\it representing} the EMV-algebra $M$.
For more details, we refer to \cite{Dvz}.

\begin{lem}\label{lem3.2}
Let $(M;\vee,\wedge,\oplus,0)$ be an EMV-algebra. For each idempotent element $a\in\mI(M)$, consider the MV-algebra $([0,a];\oplus,\lambda_a,0,a)$.
Define a binary relation $\ra_a$ on $[0,a]$ by $x\ra_a y=\lambda_a(x)\oplus y$. Then
\begin{itemize}[nolistsep]
\item[{\rm (i)}] If $a\leq b$ are elements of $\mI(M)$, then for each $x,y\leq a$, $x\ra_a y=(x\ra_b y)\wedge a$.
\item[{\rm (ii)}] $x\odot y\leq (x\odot x)\vee (y\odot y)$.
\end{itemize}
\end{lem}

\begin{proof}
(i) $x\ra_a y=\lambda_a(x)\oplus y=(\lambda_b(x)\wedge a)\oplus y=(\lambda_b(x)\oplus y)\wedge (a\oplus y)=(x\ra_b y)\wedge a$.

(ii) Let $x,y\in M$. Choose $b\in \mI(M)$ such that $x\vee y\leq b$. Then
\begin{eqnarray*}
x\odot y&=&(x\odot y)\odot b=(x\odot y)\odot\big((x\ra_b y)\vee (y\ra_b x)\big)\\
&=&\big((x\odot y)\odot (x\ra_b y)\big)\odot
\big((x\odot y)\odot (y\ra_b x) \big) \leq (x\odot x)\vee (y\odot y).
\end{eqnarray*}
\end{proof}

\begin{defn}\cite[Def 3.1]{Dvz4}
An EMV-algebra $(M;\vee,\wedge,\oplus,0)$ is called {\em locally complete} if, for each $x\in M$, there exists
$a\in\mI(M)$ such that $x\leq a$ and the MV-algebra $([0,a];\oplus,\lambda_a,0,a)$ is a complete MV-algebra.
\end{defn}

\section{Square roots of EMV-algebras}

In this section, we will study square roots on EMV-algebras which are unary operations with some special properties. We show that if an EMV-algebra has a square root, then it is unique. We present the main properties of MV-algebras with square roots.
We show that if $M$ is an EMV-algebra and $N$ is its representing EMV-algebra with top element such that $N$ has a square root, then $M$ has a square root, too. Also, we find some representations of EMV-algebras with square roots using the properties of square roots.

\begin{defn}\label{3.1}
Let $(M;\vee,\wedge,\oplus,0)$ be an EMV-algebra. A mapping $r:M\to M$ is called a {\em square root} if it satisfies the following conditions:
\begin{itemize}[nolistsep]
\item[{\rm (Sq1)}] for all $x\in M$, $r(x)\odot r(x)=x$;
\item[{\rm (Sq2)}] for each $x,y\in M$, $y\odot y\leq x$ implies $y\leq r(x)$.
\end{itemize}
An EMV-algebra $(M;\vee,\wedge,\oplus,0)$ has {\em square roots} if there exists a square root $r$ on $M$.
\end{defn}

Before we give some properties of a square root on an EMV-algebra $M$, we note that any square root $s:M\to M$ is a one-to-one map: If $s(x)=s(y)$, then $x=s(x)\odot s(x)= s(y)\odot s(y)=y$. Moreover, if
$r_1$ and $r_2$ are two square roots on an EMV-algebra $M$, then $r_1=r_2$. Indeed, by (Sq2), for each $x\in M$
$r_1(x)\odot r_1(x)\leq x$ implies that $r_1(x)\leq r_2(x)$. In a similar way, $r_2(x)\leq r_1(x)$. That is $r_1=r_2$.

\begin{prop}\label{3.2}
Let $r$ be a square root on an EMV-algebra $(M;\vee,\wedge,\oplus,0)$. Then for each $x,y\in M$ and each $a,b\in\mI(M)$, we have:
\begin{itemize}[nolistsep]
\item[{\rm (i)}] $x\leq x\vee r(0)\leq r(x)$.
\item[{\rm (ii)}] $x\leq y$ implies that $r(x)\leq r(y)$.
\item[{\rm (iii)}] $r(x)\odot r(y)\leq r(x\odot y)$.
\item[{\rm (iv)}] $x\wedge y\leq r(x)\odot r(y)$.
\item[{\rm (v)}] $r(x)\odot r(y)\leq x\vee y$ and if $x\leq a$, then $x\wedge\lam_a(x)\leq r(0)$.
\item[{\rm (vi)}] $r(x)\in\mI(M)$ \iff $r(x)=x$.
\item[{\rm (vii)}] $r(x)\wedge r(y)=r(x\wedge y)$.
\item[{\rm (viii)}] $r_b:[0,b]\to [0,b]$ defined by $r_b(x)=r(x)\wedge b$ is a
square root on the MV-algebra $[0,b]$ with $r_b(b)=b$.
\item[{\rm (ix)}] If $y\leq r(x)\odot r(y)$, then $y\leq x$.
\item[{\rm (x)}] If $r(x),r(y)\leq b$, then $r(x)\ra_b r(y)=r(x\ra_b y)\wedge b$.
\item[{\rm (xi)}] If $M$ is a generalized Boolean algebra (that is, $\oplus=\vee$), then $r$ is the identity map on $M$.
\item[{\rm (xii)}] $r(x\vee y)=r(x)\vee r(y)$ and $r(x\odot y)=(r(x)\odot r(y))\vee r(0)$. Consequently, if $r(0)\leq x$, then $r(x\odot x)=x$.
\item[{\rm (xiii)}] For each $a\in \mI(M)$ with $r(r(0))\leq a$, the element
$(r(0)\ra_a 0)\odot (r(0)\ra_a 0)$ is an idempotent element of $M$.
\item[{\rm (xiv)}] $x\leq r(x\odot x)$ and $r(x\odot x)\odot r(x\odot x)=r(x)\odot r(x)\odot r(x)\odot r(x)$.
\end{itemize}
\end{prop}

\begin{proof}
(i) $x\leq r(x)$ follows from (Sq1). Also, $(x\vee r(0))\odot (x\vee r(0))=(x\odot x)\vee (r(0)\odot x)\vee (x\odot r(0))\vee (r(0)\odot r(0))=
(x\odot x)\vee (r(0)\odot x)=x\odot (x\vee r(0))\leq x$ and (Sq2) imply that $x\vee r(0)\leq r(x)$.

(ii) By (Sq2) and  $r(x)\odot r(x)=x\leq y$, we have $r(x)\leq r(y)$.

(iii) From $(r(x)\odot r(y))\odot (r(x)\odot r(y))=(r(x)\odot r(x))\odot (r(y)\odot r(y))=x\odot y$ and
(Sq2) it follows that $r(x)\odot r(y)\leq r(x\odot y)$.

(iv) By (ii), $x\wedge y=r(x\wedge y)\odot r(x\wedge y)\leq r(x)\odot r(y)$.

(v) By Lemma \ref{lem3.2}, we know that $x\odot y\leq (x\odot x)\vee (y\odot y)$. So
$$r(x)\odot r(y)\leq (r(x)\odot r(x))\vee (r(y)\odot r(y))=x\vee y.$$

Now, let $x\le a$. $(x\wedge\lam_a(x))\odot (x\wedge\lam_a(x))=\big((x\wedge\lam_a(x))\odot x\big)\wedge \big((x\wedge\lam_a(x))\odot \lam_a(x) \big)\leq
(\lam_a(x)\odot x)\wedge (x\odot \lam_a(x))=0$, so by
(Sq2), $x\wedge\lam_a(x)\leq r(0)$.

(vi) If $r(x)\in\mI(M)$, then by (Sq1), $x=r(x)\odot r(x)=r(x)\in\mI(M)$. Conversely, if $r(x)=x$, then
$x=r(x)\odot r(x)\leq r(x)=x$ and so $r(x)\in\mI(M)$.

(vii) From (ii) and $x\wedge y\leq x,y$, we have that $r(x\wedge y)\leq r(x)\wedge r(y)$. Also,
$\big(r(x)\wedge r(y) \big)\odot \big(r(x)\wedge r(y) \big)=(r(x)\odot r(x))\wedge (r(x)\odot r(y))\wedge
(r(y)\odot r(x))\wedge (r(y)\odot r(y))=x\wedge (r(x)\odot r(y))\wedge y\leq x\wedge y$, so by (Sq2),
$r(x)\wedge r(y)\leq r(x\wedge y)$.

(viii) Clearly, $r_b:[0,b]\to [0,b]$ is well-defined. For each $x,y\in [0,b]$, we have
$r_b(x)\odot r_b(x)=(r(x)\wedge b)\odot (r(x)\wedge b)=(r(x)\odot r(x))\wedge (r(x)\odot b)\wedge b=r(x)\odot r(x)=x$.
Moreover, if $y\odot y\leq x$, then $y\leq r(x)$ consequently $y\leq r(x)\wedge b=r_b(x)$.

(ix) Let $b\in \mI(M)$ with $r(x\vee y)\leq b$. Consider the binary operation $\ra_b$ defined in Lemma \ref{lem3.2}.
In the MV-algebra $[0,b]$ containing $r(x)$ and $r(y)$ we have
\begin{eqnarray*}
y&=& y\wedge (r(x)\odot r(y))=r(y)\odot r(y)\wedge (r(x)\odot r(y))=r(y)\odot (r(x)\wedge r(y))\\
&=& r(y)\odot \big(r(y)\odot (r(y)\ra_b r(x)) \big)=y\odot (r(y)\ra_b r(x))\leq r(x)\odot r(y)\odot (r(y)\ra_b r(x))\\
&\leq & r(x)\odot r(x)=r(x).
\end{eqnarray*}
(x) In the MV-algebra $[0,b]$, we have
$x\odot (r(x)\ra_b r(y))\odot (r(x)\ra_b r(y))=r(x)\odot r(x)\odot (r(x)\ra_b r(y))\odot (r(x)\ra_b r(y))\leq
r(y)\odot r(y)=y$. It follows that $(r(x)\ra_b r(y))\odot (r(x)\ra_b r(y))\leq x\ra_b y$ and so by
(Sq1), $r(x)\ra_b r(y)\leq r(x\ra_b y)$. That is $r(x)\ra_b r(y)\leq r(x\ra_b y)\wedge b$. Conversely, by (ii) and (iii),
$r(x)\odot r(x\ra_b y)\leq r(x\odot (x\ra_b y))\leq r(y)$, so
$r(x)\odot (r(x\ra_b y)\wedge b)\leq r(y)$ which implies that $r(x\ra_b y)\wedge b\leq r(x)\ra_b r(y)$.
Therefore, $r(x)\ra_b r(y)=r(x\ra_b y)\wedge b$.

(xi) For each $x\in M$, we have $r(x)=r(x)\odot r(x)=x$.

(xii) Let $x,y\in M$ and $a\in\mI(M)$ be such that $r(x\vee y)\leq a$. By (viii), $r_a:[0,a]\ra [0,a]$ is a square root on the MV-algebra $[0,a]$,
so from \cite[Cor 2.13 (xv)]{Hol} it follows that $r_a(x\vee y)=r_a(x)\vee r_a(y)$. Hence,
$r(x\vee y)=r(x\vee y)\wedge a=(r(x)\wedge a)\vee (r(y)\wedge a)=(r(x)\vee r(y))\wedge a=r(x)\vee r(y)$.
Also, by \cite[Prop 2.17 (xxviii)]{Hol}, we get that
$r_a(x\odot y)=(r_a(x)\odot r_a(y))\vee r_a(0)$. Now, $r(x),r(y),r(0)\leq r(x\vee y)=a$ implies that
$r_a(x\odot y)=r(x\odot y)$, $r_a(x)=r(x)$, $r_a(y)=r(y)$ and $r_a(0)=r(0)$. That is,
$r(x\odot y)=(r(x)\odot r(y))\vee r(0)$.

(xiii) Let $a\in \mI(M)$ be such that $r(r(0))\leq a$. Then by (x), $r(r(0)\ra_a 0)=r(r(0))\ra_a r(0)=
r(r(0))\ra_a \big(r(r(0))\odot r(r(0))\big)=\big(r(r(0))\ra_a 0 \big)\vee r(r(0))\leq \big(r(0)\ra_a 0 \big)\vee r(r(0))$.
So,
\begin{equation}
\label{Eer43} (r(0)\ra_a 0)\odot (r(0)\ra_a 0)=(r(r(0)\ra_a 0))^4\leq ((r(0)\ra_a 0)\vee r(r(0)))^4.
\end{equation}
On the other hand, by Lemma \ref{lem3.2}(ii), we can easily show that
\begin{equation}
((r(0)\ra_a 0)\vee r(r(0)))^2\leq ((r(0)\ra_a 0)^2 \vee\big( r(r(0))\odot r(r(0))\big)=
(r(0)\ra_a 0)^2 \vee r(0)
\end{equation}
and similarly,
$\big((r(0)\ra_a 0 )\vee r(r(0))\big)^4=(r(0)\ra_a 0 )^4 \vee (r(0)\odot r(0))=(r(0)\ra_a 0 )^4$.
Therefore, $(r(0)\ra_a 0)\odot (r(0)\ra_a 0)$ is an idempotent element of $M$.

(xiv) $x\odot x\leq x\odot x$ and (Sq2) imply that $x\leq r(x\odot x)$. The second part follows from
(Sq1): Indeed,
$r(x\odot x)\odot r(x\odot x)=x\odot x=r(x)\odot r(x)\odot r(x)\odot r(x)$.
\end{proof}

The next remark helps us to find some upper and lower bounds for $r(x)$, where $r$ is a square root on an EMV-algebra $M$.

\begin{rmk}\label{rmk-cor}
Let $r$ be a square root on an EMV-algebra $(M;\vee,\wedge,\oplus,0)$.

{\rm (i)} In {\rm Proposition \ref{3.2} (i)} and {\rm(xiv)}, we saw that $x\leq x\vee r(0)\leq r(x)$ and $x\leq r(x\odot x)$. Now, by
part {\rm(xii)}, we have $r(x\odot x)=(r(x)\odot r(x))\vee r(0)=x\vee r(0)$. Therefore, for each non-idempotent element $x\in M$
$x\vee r(0)<r(x)$, otherwise, $r(x)=r(x\odot x)$ which implies that $x=x\odot x$ (contradicts with the assumption).

{\rm (ii)} Choose $x\in M$. By definition, $r(x)$ is the $\max\{z\in M\mid z\odot z\leq x\}$.
We claim that $x\oplus r(0)$ is an upper bound for $r(x)$ and $\{z\in M\mid z\odot z\leq x\}$.

Indeed, first, in each EMV-algebra, for all $x,y,z$ the following inequality holds:
\begin{eqnarray*}
(x\ominus z)\odot (y\ominus z)&=&\left(x\odot \lambda_a(z)\right)\odot \left(y\odot \lambda_a(z)\right),\quad \mbox{ $x,y,z\leq a\in\mI(M)$ }\\
&=& (x\odot y)\odot \left(\lambda_a(z)\odot \lambda_a(z)\right)=(x\odot y)\odot \lambda_a(z\oplus z)\\
&=& (x\odot y)\ominus(z\oplus z)\leq (x\odot y)\ominus z.
\end{eqnarray*}

Second, if $z\in M$ such that $z\odot z\leq x$, then $(z\odot z)\ominus x=0$ which means
$(z\ominus x)\odot (z\ominus x)=0$, consequently $z\ominus x\leq r(0)$ and so $z\leq x\oplus r(0)$.
Therefore, from {\rm (i)} and {\rm (ii)} we get $x\leq x\vee r(0)=r(x\odot x)\leq r(x)\leq x\oplus r(0)$.

{\rm (iii)} If $x\in M$ is idempotent, then by (ii) and {\rm Proposition \ref{3.2} (i)} and {\rm(xiv)}, we have
$x\vee r(0)\leq r(x)\leq x\oplus r(0)=x\vee r(0)$ which means $r(x)=x\vee r(0)$.
\end{rmk}

\begin{exm}\label{3.4}
(i) If $(M;\vee,\wedge,\oplus,0)$ is a generalized Boolean algebra, then the identity map $\id_M:M\ra M$ is a square root.
Indeed, for each $x\in M$ we have $\id_M(x)\odot \id_M(x)=x\odot x=x$. Similarly, (Sq2) holds. Also, $\id_M$ is the only
square root on $M$.

(ii) Consider the real unit interval $[0,1]$ with the standard operation $x\oplus y=(x+y)\wedge 1$ and $x'=1-x$, which is a complete MV-algebra.
The map $r:[0,1]\to [0,1]$ sending $x$ to $r(x)=\bigvee\{z\in M\mid z\odot z\leq x\}$ is a square root on $[0,1]$.
Easy calculations show that (a) $r(x)=(1+x)/2$, $x\in [0,1]$, (b) $r(0)=1/2$, $r(1)=1$, (c) $r(x)<r(0) \oplus x$ for $x\in [0,1)$.

The example is a particular case of Proposition \ref{pr:exist}.

(iii) Let $M_1$ be a generalized Boolean algebra that has no top element and let $M_2$ be the MV-algebra of the real interval $[0,1]$. Then $M=M_1\times M_2$ is a proper EMV-algebra which has a square root $r(x_1,x_2)=(x_1,(x_2+1)/2)$, $(x_1,x_2)\in M$. This is a particular case of a general situation of an EMV-algebra described in Corollary \ref{th:EMV} below.

(iv) Let $\{M_i\mid i\in I\}$ be a family of MV-algebras and $r_i:M_i\to M_i$ be a square root on $M_i$ for each $i\in I$. Then the direct product $\prod_{i\in I}M_i$ has a square root, namely $r((x_i)_i)=(r_i(x_i))_i$, $x= (x_i)_i$.

Consider the EMV-algebra $\sum_{i\in I}M_i=\{x=(x_i)_{i\in I}\mid |\{i\in I\mid x_i\ne 0\}|<\infty\}$  (see \cite[Exm 3.2(6)]{Dvz}). Then not necessarily $\sum_{i\in I}M_i$ has square roots.
Indeed, let $r_i(0_i)>0$ for all but finitely many indices $i\in I$. Then $\sum_{i\in I}M_i$ has no square root. Suppose the converse, let $s$ be a square root on $\sum_{i\in I}M_i$. Define $s_i:M_i\to M_i$ by $s_i(y_i)=\pi_i(\ldots,0,\ldots,y_i,\ldots,0,\ldots)$ for each $i$. Then $s_i$ is a square root on $M_i$ so that $s_i=r_i$. But $s(0)=(s_i(0_i))_i= (r_i(0_i))_i\notin \sum_{i\in I}M_i$, where $0=(0_i)_i$.

In other words, a subalgebra of an EMV-algebra with square root does not necessary have a square root.

(v) The Chang MV-algebra $C=\Gamma(\mathbb Z\lex \mathbb Z,(1,0))$ has no square root; the set $\{z\in C\mid z\odot z\le (0,n)\}$ has no maximum for each integer $n\ge 0$.

(vi) Let $B$ be the set of all finite subsets of $\mathbb N$. Then $B$ is a generalized Boolean algebra which is an
EMV-algebra. Consider the EMV-algebra $M=B\times C$, where $C$ is the Chang MV-algebra.
Then, $M$ is a proper EMV-algebra with no square root.
Indeed, if $r:M\to M$ is a square root, then the set $\{(x,y)\in M\colon (x,y)\odot (x,y)=(0,0)\}$ must have a greatest element $(u,v)$.
Clearly, $u=0$ and so $v=\max\{y\in C\colon y\odot y=0\}$ which is absurd, since, in the Chang MV-algebra, this set does not have a maximum.

(vii) An example of a locally complete EMV-algebra without square root is in Remark \ref{rm:contra1} and examples of finite (locally complete) EMV-algebras without square roots are given in Remark \ref{rm:contra2}.

(viii) Let $M=\{i/2^n\mid i=0,\ldots, 2^n,\, n\ge 1\}$ be the MV-algebra of dyadic numbers in the real interval [0,1]. It is not locally complete, but it has a square root $s(x)=(x+1)/2$, $x \in M$ (the restriction of the square root (ii) on $[0,1]$, see \cite{Hol,NPM}. Similarly, the MV-algebra of rational numbers in $[0,1]$ has the square root, the restriction of (ii).

(ix) There are uncountably many MV-subalgebras of $[0,1]$ having no square roots, see Example \ref{ex:irr}(1).

(x) There are countably many MV-subalgebras of $[0,1]$ with square roots, see Example \ref{ex:irr}(3).
\end{exm}

\begin{prop}\label{pr:3.5}
Let $r$ be a square root on an EMV-algebra $(M;\vee,\wedge,\oplus,0)$.
If $a\in \mI(M)$ is such that $r(r(0))\leq a$, then
\begin{itemize}[nolistsep]
\item[{\rm (i)}] $r(x\ra_a 0)=r(x)\ra_a r(0)$;
\item[{\rm (ii)}] $r(x\oplus y)=\left(r(x)\odot \lam_a(r(0))\right)\oplus r(y)$.
\end{itemize}
 Consequently, if $M$ has a top element, then $r(x\oplus y)=(r(x)\odot r(0)')\oplus r(y)$.
\end{prop}

\begin{proof}
(i) Let $x\in M$, $a\in \mI(M)$ and $r(r(0)),x\leq a$. Then by Remark \ref{rmk-cor},
$r(x\ra_a 0)\leq r(a)\leq r(0)\oplus a=a$, so $r(x)\ra_a r(0)=r(x\ra_a 0)\wedge a=r(x\ra_a 0)$.  \\
(ii) Let $x,y\in M$, $a\in \mI(M)$ and $r(r(0)),x,y\leq a$.
It follows that
\begin{eqnarray*}
r(x\oplus y)&=& r\left(\lam_a\left(\lam_a(x)\odot \lam_a(y)\right)\right)=r(\lam_a(x)\odot \lam_a(y))\ra_a r(0), \mbox{ by (i)}\\
&=& \left(\left(r(\lam_a(x))\odot r(\lam_a(y))\right)\vee r(0)\right)\ra_a r(0), \mbox{ by (xii)}\\
&=& \left(r(\lam_a(x))\odot r(\lam_a(y))\right)\ra_a r(0)\\
&=& \left(\left(r(x)\ra_a r(0)\right)\odot \left(r(y)\ra_a r(0)\right)\right)\ra_a r(0), \mbox{ by (i)}\\
&=& (r(x)\ra_a r(0))\ra_a \left(\left(r(y)\ra_a r(0)\right)\ra_a r(0) \right)\\
&=& (r(x)\ra_a r(0))\ra_a\big(r((y\ra_a 0)\ra_a 0)\big), \mbox{ by (i)}\\
&=& (r(x)\ra_a r(0))\ra_a r(y)=\big(r(x)\odot \lam_a(r(0))\big)\oplus r(y).
\end{eqnarray*}
The proof of the rest is straightforward.
\end{proof}

Proposition \ref{pr:3.5}(ii) implies, for each square root $r$ on an EMV-algebra $M$, $r(x\oplus y)\leq r(x)\oplus r(y)$, $x,y \in M$.

\begin{thm}\label{3.5}
Let $r$ be a square root on an EMV-algebra $(M;\vee,\wedge,\oplus,0)$. Then $M$ is a generalized Boolean algebra \iff $r(0)=0$.

In addition, the restriction of $r$ onto the generalized Boolean algebra $\mathcal I(M)$ is a square root on $\mathcal I(M)$ \iff $r(0)=0$.
\end{thm}

\begin{proof}
Let $r(0)=0$ and $u\in \mathcal I(M)$. We will show that $[0,u]$ is a Boolean algebra. Since it is an MV-algebra, it suffices to prove that
$x\wedge (x\ra_u 0)=0$ for all $x\in [0,u]$ (see \cite[Thm 1.5.3]{CDM}).
By Proposition \ref{3.2}(viii), $r_u:[0,u]\ra [0,u]$ is a square root on the MV-algebra $[0,u]$. It follows from \cite[Prop 2.11(xxi)]{Hol}
that $x\wedge (x\ra_u 0)=r_u(x)\odot r_u(x\ra_u 0)= r_u(x)\odot (r_u(x)\ra_u r_u(0))\leq r_u(0)=r(0)\wedge u=0$.

The proof of the converse follows from Proposition \ref{3.2}(xi).

The second statement is a direct corollary of the first one.
\end{proof}

\begin{prop}\label{Homo}
Let $f:M\to E$ be a homomorphism of EMV-algebras and $r$ be a square root on $M$. Then
$t:\im(f)\to \im(f)$ defined by $t(f(x))=f(r(x))$ for all $x\in M$ is a square root on $\im(f)$.
\end{prop}

\begin{proof}
(i)  For each $x\in M$, $f(r(x))\odot f(r(x))=f(r(x)\odot r(x))=f(x)$.

(ii) Let $x,y\in M$ be such that $f(y)\odot f(y)\leq f(x)$. Choose $a\in\mI(M)$ such that $x,y,r(0)\leq a$. We note that in any MV-algebra $u\le v$ iff $u'\oplus v=1$. Consider the MV-algebra $[0,a]$. Then
$f(y)\odot f(y)\leq f(x)$ implies that $f(\lam_a(y\odot y)\oplus x)=\lam_{f(a)}(f(y\odot y))\oplus f(x)=f(a)$ because $f(a)$ is the top element in the MV-algebra $[0,f(a)]$.
Since $r_a$ is a square root on $[0,a]$ (see Proposition \ref{3.2}(viii)), we have
\begin{eqnarray*}
\lam_a(y\odot y)\oplus x&\leq & r_a(\lam_a(y\odot y)\oplus x)=r_a((y\odot y)\ra_a x)\\
&=& r_a(y\odot y)\ra_a r_a(x), \mbox{ by Proposition \ref{3.2}(x)}\\
&=& \lam_a(r_a(y\odot y))\oplus r_a(x)=\lam_a\big((r_a(y)\odot r_a(y))\vee r_a(0)\big)\oplus r_a(x)\\
&=& \lam_a(y\vee r(0))\oplus r_a(x)\leq \lam_a(y)\oplus r_a(x).
\end{eqnarray*}
It follows that $f(a)=f(\lam_a(y\odot y)\oplus x)\leq f(\lam_a(y)\oplus r_a(x))\leq f(a)$ which means that
$f(y)\leq f(r_a(x))\le f(r(x))$. That is, $f(y)\odot f(y)\le f(x)$ implies $f(y)\le f(r(x))$.

(iii) We show that $f(x)=f(y)$ entails $f(r(x))=f(r(y))$.
We have that $f(r(x))\odot f(r(x))= f(r(x)\odot r(x))=f(x)=f(y)$ and so by (ii), $f(r(x))\leq f(r(y))$. In a similar way, $f(r(y))\leq f(r(x))$. Whence, $f(x)=f(y)$ implies $f(r(x))=f(r(y))$.

Thus, $\tau:\im(f)\to \im(f)$ sending $f(x)$ to $f(r(x))$ is a square root on $\im(f)$.
\end{proof}

\begin{thm}\label{3.6}
Let $(M;\vee,\wedge,\oplus,0)$ be an EMV-algebra  and $(N;\vee,\wedge,\oplus,0)$ be its representing EMV-algebra with top element.
If $R:N\ra N$ is a square root on $N$, then $r:=R\big|_{M}$ is a square root on $M$.
\end{thm}

\begin{proof}
It suffices to show that $R(M)\s M$. Indeed, let $x\in N$.

(i) If $x\in M'$, then from $x\leq R(x)$ it follows that $R(x)\notin M$, that is $R(x)\in M'$.

(ii) If $x\in M$, then $R(x)\in M$, otherwise, $R(x)=y'$ for some $y\in M$, so
$x=R(x)\odot R(x)=y'\odot y'=(y\oplus y)'\in M'$ which is a contradiction. Note that $M$ is closed under the operation $\oplus$.

Therefore, $R(M)\s M$, which means that $r:M\to M$ is a square root.
\end{proof}

\begin{thm}\label{3.7}
Let $(B;\vee,\wedge,0)$ be a generalized Boolean algebra and $a\in B\setminus \{0\}$. Then $B\cong B_1\times B_2$, where
$B_1=[0,a]$ and $B_2=\bigsqcup_{a\leq b} [0,\lambda_b(a)]$ \rm{(for notation, see \cite[Page 886]{Dvz4})}.
\end{thm}

\begin{proof}
We know that $[0,a]$ and $[0,\lambda_b(a)]$ are generalized Boolean algebras. Also,
$$
\{([0,\lambda_b(a)];\oplus,\lambda_{_{\lambda_b(a)}},0,\lambda_b(a))\mid b\geq a\}
$$
is a family of nested MV-algebras.
By \cite[Page 886]{Dvz4}, $\bigsqcup_{a\leq b}[0,\lambda_b(a)]$ is an EMV-algebra such that its elements are idempotent,
that is $M_2:=\bigsqcup_{a\leq b}[0,\lambda_b(a)]$ is a generalized Boolean algebra.

Define $\varphi:B\to M_1\times M_2$ by $\varphi(x)=(a\wedge x,\lambda_b(a)\wedge x)$, where $b$ is an arbitrary element of $B$ such that
$a,x\leq b$. First, we show that $\varphi$ is well-defined.

(i) Let $b_1,b_2\in B$ such that $x,a\leq b_1,b_2$. Set $b:=b_1\vee b_2$. By Proposition \ref{2.2}(i),
$\lambda_b(a)\wedge x=\lambda_b(a)\wedge (b_1\wedge x)=(b_1\wedge \lambda_b(a))\wedge x=\lambda_{b_1}(a)\wedge x$.
In a similar way, $\lambda_b(a)\wedge x=\lambda_{b_2}(a)\wedge x$. That is, $\varphi$ is well-defined. Now, we prove that $\varphi$ preserves the
operations $\vee$, $\wedge$, and $0$.

(ii) By definition, $\varphi$ preserves $0$. Let $x,y\in B$. Choose $b\in B$ such that $x,y,a\leq b$. Then by (i), we have
\begin{eqnarray*}
\varphi(x\vee y)&=& (a\wedge(x\vee y),\lambda_b(a)\wedge (x\vee y))=((a\wedge x)\vee (a\wedge y), (\lambda_b(a)\wedge x)\vee (\lambda_b(a)\wedge y))\\
&=& ((a\wedge x), (\lambda_b(a)\wedge x))\vee ((a\wedge y), (\lambda_b(a)\wedge y)).
\end{eqnarray*}
On the other hand, since $x,y,a\leq b$, by (i) we get
\begin{eqnarray*}
\varphi(x)&=& (a\wedge x,\lambda_b(a)\wedge x),\\
\varphi(y)&=&(a\wedge y, \lambda_b(a)\wedge y).
\end{eqnarray*}
Hence, $\varphi(x\vee y)=\varphi(x)\vee \varphi(y)$. In a similar way, we can show that $\varphi$ preserves $\wedge$.

(iii) The mapping $\varphi$ is a homomorphism of EMV-algebras.
By (ii) and \cite[Rem 3.10]{Dvz}, it suffices to show that, for each $x\in B$, there exists $u\in B$ such that $\varphi\big|_{[0,u]}:[0,u]\to [0,\varphi(u)]$ is a homomorphism of Boolean algebras. That is, for each $z\leq u$,
$\varphi(\lambda_u(z))=\lambda_{\varphi(u)}(\varphi(z))$.
Let $a\leq u\in B$ and $b$ be an element of $B$ such that $u,a\leq b$. Then $\varphi(u)=(a\wedge u,\lambda_b(a)\wedge u)$.
By definition, $\varphi(u)=(a\wedge u,\lambda_b(a)\wedge u)=(a\wedge u,\lambda_u(a))$. Thus, for each $x\in [0,u]$,
\begin{eqnarray}
\lambda_{\varphi(u)}(\varphi(x))&=&\lambda_{\varphi(u)}(a\wedge x,\lambda_b(a)\wedge x)=(\lambda_{a\wedge u}(a\wedge x),\lambda_{\lambda_u(a)}(\lambda_b(a)\wedge x)).
\end{eqnarray}
Since $a\wedge x\leq a\wedge u\leq a\leq u$, by Proposition \ref{2.2}(i),
$\lambda_{u\wedge a}(a\wedge x)=\lambda_{u}(a\wedge x)\wedge a=(\lambda_u(a)\vee \lambda_u(x))\wedge a=\lambda_u(x)\wedge a$. Moreover,
$\lambda_{\lambda_u(a)}(\lam_b(a)\wedge x)=\lambda_{b}(\lam_b(a)\wedge x)\wedge \lambda_u(a)=
(a\vee \lambda_b(x))\wedge \lambda_u(a)=\lambda_b(x)\wedge \lambda_u(a)=\lambda_b(x)\wedge (\lambda_b(a)\wedge u)=
(\lambda_b(x)\wedge u)\wedge \lambda_b(a)=\lambda_u(x)\wedge \lambda_b(a)$.
So,
\begin{eqnarray*}
\varphi(\lambda_u(x))&=& (a\wedge \lambda_u(x),\lambda_b(a)\wedge \lambda_u(x))=\lambda_{\varphi(u)}(\varphi(x)).
\end{eqnarray*}
From (i)--(iii), it follows that $\varphi$ is a homomorphism of EMV-algebras. Clearly, $\varphi$ is one-to-one: Indeed,
for each $x,y\in B$, $\varphi(x)=\varphi(y)$ implies that (assume that $x\vee y\leq b\in B$)
\begin{eqnarray*}
(a\wedge x,\lambda_b(a)\wedge x)=\varphi(x)=\varphi(y)=(a\wedge y,\lambda_b(a)\wedge y),
\end{eqnarray*}
consequently, $x=x\wedge b=x\wedge (a\vee \lambda_b(a))=(a\wedge x)\vee (\lambda_b(a)\wedge x)=
(a\wedge y)\vee (\lambda_b(a)\wedge y)=y\wedge (a\vee \lambda_b(a))=y\wedge b=y$.
Now, let $y=(y_1,y_2)\in B_1\times B_2$. There exists $b\in B$ such that $a\leq b$, $y_1\leq a$ and $y_2\leq \lambda_{b}(a)$.
Set $x:=y_1\vee y_2$. We can easily show that $\varphi(x)=(a\wedge x,\lambda_{b}(a)\wedge x)=(y_1,y_2)=y$.
Therefore, $\varphi:B\to B_1\times B_2$ is an isomorphism and $B\cong B_1\times B_2$.
\end{proof}

Similarly to the proof of Theorem \ref{3.7}, we can generalize this result for EMV-algebras instead of generalized Boolean algebras.

\begin{cor}\label{3.8}
Let $(M;\vee,\wedge,\oplus,0)$ be an EMV-algebra and $a\in\mI(M)\setminus \{0\}$. Then
$M_1:=[0,a]$,  $M_2:=\bigsqcup_{a\leq b\in\mI(M)}[0,\lambda_b(a)]$ are EMV-algebras, and
$\varphi:M\to M_1\times M_2$ sending $x$ to $\varphi(x):=(x\wedge a,x\wedge \lambda_b(a))$ where $x\leq b$ and $a\leq b\in\mI(M)$,
is an isomorphism of EMV-algebras.
\end{cor}

We note that, in Theorem \ref{3.7} and Corollary \ref{3.8}, the homomorphism $\varphi$ is indeed a strong homomorphism (see \cite[Page 122]{Dvz}).

\section{Classification of EMV-algebras with square roots}

We introduce strict EMV-algebras and we show that if an EMV-algebra is strict, then it has a top element. Strict EMV-algebras will serve for a classification of EMV-algebras with square roots.
We show that each EMV-algebra with square root is a generalized Boolean algebra or a strict EMV-algebra or a direct product of a generalized Boolean algebra and a strict EMV-algebra.

\begin{defn}\label{3.9}
An EMV-algebra $(M;\vee,\wedge,\oplus,0)$ with the square root $s:M\to M$ is called {\em strict} if, for each $b\geq s(0)$, the MV-algebra
$([0,b];\oplus,\lambda_b,0,b)$ with the square root $s_b$ is strict, or, equivalently, $s_b(0)=\lambda_b(s_b(0))$.
\end{defn}

In the sequel, we will propose a representation for EMV-algebras using square roots.

\begin{thm}\label{3.10}
Each strict EMV-algebra has a top element.
\end{thm}

\begin{proof}
Let $s$ be a strict square root on EMV-algebra $(M;\vee,\wedge,\oplus,0)$. Let $a\in\mI(M)$ be such that $s(0)\leq a$. Then $s_a(0)=s(0)\wedge a=s(0)=\lambda_a(s(0))$. We claim that $a$ is a top element of $M$. Choose $a\leq b\in\mI(M)$. By the assumption, $s_b(0)=s(0)\wedge b=s(0)=\lambda_b(s(0))$, hence
 Proposition \ref{2.2} implies
$s(0)=\lambda_b(s(0))=\lambda_a(s(0))\vee \lambda_b(a)=s(0)\vee \lambda_b(a)$. That is, $\lambda_b(a)\leq s(0)$. Also
$b=a\vee\lambda_b(a)\leq a\vee s(0)\leq a\vee a=a$. Therefore, $a$ is the top element of $M$.
\end{proof}

\begin{cor}\label{strict-N}
{\rm (i)} Let $s$ be a square root on an EMV-algebra $(M;\vee,\wedge,\oplus,0)$. If $s(0)\leq b\in\mI(M)$ and $s_b$ is strict, then for each $a\in \mathcal I(M)$ such that $s(0)\leq a<b$, the square root $s_a$ on the MV-algebra $[0,a]$ cannot be strict. Specially, if $M$ has a top element $1$ and $s$ is a strict square root on the MV-algebra $(M;\oplus,\lam_1,0,1)$, then the only idempotent element $a$ of $M$ with $s(0)\leq a$ is $1$.

{\rm (ii)} Each strict EMV-algebra is a strict MV-algebra.

{\rm (iii)} Let $M$ be an EMV-algebra with a square root $r$ and  let $N$ be its representing EMV-algebra with top element. If $N$ is strict, then $M$ is strict, too.

{\rm (iv)} The homomorphic image of a strict EMV-algebra is also a strict EMV-algebra.
\end{cor}

\begin{proof}
For (i), we prove only its second part. Let $M$ have a top element $1$ and let $s$ be a strict square root on the MV-algebra $(M;\oplus,\lam_1,0,1)$
(we use $x'$ instead of $\lam_1(x)$ for all $x\in M$).
Let $a$ be an idempotent element of $M$ such that $s(0)\leq a<1$.
By Proposition \ref{3.2}(viii), $s_a:[0,a]\to [0,a]$ sending $x$ to $s(x)\wedge a$ is strict. Then we have
$s(0)=s_a(0)=\lam_a(s_a(0))=\lam_a(s(0))$. By the assumption, $s(0)'=s(0)$ which implies that
$s(0)=s(0)'=\lam_a(s(0))\vee a'=s(0)\vee a'$, so that $a'\leq s(0)$. On the other hand,
$1=a\vee a'\leq a\vee s(0)\leq a$ which is a contradiction.

(ii) The proof of the second part is clear by (i).

(iii) Let $N$ be strict. Then there is a square root $s:N\to N$ such that $s(0)'=s(0)$. Since $s\big|_M:M\to M$ is a square root
on $M$, by the note just before, Proposition \ref{3.2}, $r=s\big|_M$. If $N=M$, then the proof is complete.
Suppose that $N\neq M$. Since $N=M\cup M'$ and $M\cap M'=\emptyset$ (otherwise, $1\in M$ which implies that $N=M$),
$s(0)$ can not belong to $M\cup M'=N$, that is a contradiction. Therefore, $M=N$ and $M$ is strict.

(iv) By Proposition \ref{Homo}, the homomorphic image of an EMV-algebra $M$ with square roots is again an EMV-algebra with square roots. By (ii), $M$ is with top element so it is equivalent to an MV-algebra with square roots. We get the statement by applying \cite[Thm 3.2]{Amb}.
\end{proof}

\begin{thm}\label{3.11}
Let $s:M\to M$ be a square root on an EMV-algebra $(M;\vee,\wedge,\oplus,0)$. Then only one of the following statements holds:
\begin{itemize}[nolistsep]
\item[{\rm (i)}] The EMV-algebra $M$ is a generalized Boolean algebra.
\item[{\rm (ii)}] The EMV-algebra $M$ is a strict EMV-algebra.
\item[{\rm (iii)}] The EMV-algebra $M$ is isomorphic to the direct product $M_1\times M_2$, where $M_1$ is a generalized Boolean algebra and $M_2$ is a strict EMV-algebra. 
\end{itemize}
\end{thm}

\begin{proof}
If $s(0)=0$, then by Theorem \ref{3.5}, $M$ is a generalized Boolean algebra. So, let $s(0)\neq 0$.

Case 1. For each $b\in\mI(M)$ with
$s(0)\leq b$, we have $\lambda_b(s(0))=s(0)$, then $M$ is strict.

Case 2. There exists an idempotent element $a\in \mI(M)$ such that
$s(0)\leq a$ and $\lambda_a(s(0))\neq s(0)$, which means that $([0,a];\oplus,\lambda_a,0,a)$ with the square root $s_a$ is not strict.
By, \cite[Thm 2.21 and its proof]{Hol}, there is $t^a\in\mI(M)$ such that $t^a\leq a$, $[0,a]\cong [0,t^a]\times [0,\lam_a(t^a)]$, where
$[0,t^a]$ is a Boolean algebra and $[0,\lam_a(t^a)]$ is a strict MV-algebra with the square root $s_{\lam_a(t^a)}$.
In the proof of \cite[Thm 2.21]{Hol}, it was proved that
$t^a:=\lam_a(s(0))\odot \lam_a(s(0))$.

For each $a\leq b\in\mI(M)$, the MV-algebra $[0,b]$ with the square root $s_b$ is neither a Boolean algebra nor a strict MV-algebra. Indeed, if it is a Boolean algebra, then $s_b(0)=0$ implies that $s_b(0)=s(0)=s(0)\wedge a=s_a(0)=0$ which is a contradiction.
Otherwise, if it is a strict MV-algebra, then $\lam_a(s_a(0))=\lam_b(s_b(0))\wedge a=s(0)\wedge a=s_a(0)$ which is also a contradiction,
(note that $s(0)\leq a\leq b$, so $s_b(0)=s_a(0)=s(0)$). So, there exists $t^b\leq b$ such that $[0,b]\cong [0,t^b]\times [0,\lam_b(t^b)]$, where
$[0,b]$ is a Boolean algebra and $[0,\lam_b(t^b)]$ is strict. Note that $a\leq b$ implies that $t^a\leq t^b$.
Moreover, $\lam_a(t^a)=\lam_a\big(\lam_a(s(0))\odot \lam_a(s(0))\big)=s(0)\oplus s(0)$ and $\lam_b(t^b)=\lam_b\big(\lam_b(s(0))\odot \lam_b(s(0))\big)=s(0)\oplus s(0)$ (since $s(0)\leq a\leq b$).

Clearly, $\{([0,t^a];\oplus,\lambda_{t^a},0,t^a)\mid a\leq b\in\mI(M)\}$ is a family of nested MV-algebras.
Set $M_1:=\bigsqcup_{a\leq b\in \mI(M)}[0,t^b]$. By \cite[Sec. 3]{Dvz4}, it is an EMV-algebra which is a
generalized Boolean algebra (since each element of $M_1$ is idempotent). Now, set $M_2:=[0,\lambda_a(t^a)]$, which is strict.

Define $\varphi:M\to M_1\times M_2$ by $\varphi(x)=(x\wedge t^b,x\wedge \lambda_a(t^a))$ where $x,a\leq b\in\mI(M)$.
Let $x\in M$ and $b,c\in\mI(M)$ such that $a,x\leq b,c$. 
By the first part of the proof, $\lambda_b(t^b)=\lambda_c(t^c)=\lambda_a(t^a)$.
Since
$x\wedge (t^c\vee \lambda_b(t^b))=x\wedge (t^c\vee \lambda_c(t^c))=x=x\wedge (t^b\vee \lambda_b(t^b))$
and $\lambda_b(t^b)$ are disjoint with $t^c$ and $t^b$, we get $x\wedge t^c=x\wedge t^b$ which means $\varphi$ is well-defined.
Similarly to the proof of Theorem \ref{3.7}, we can show that $\varphi$ is an isomorphism.
\end{proof}

In Proposition \ref{pr:uniq}, we show that in the case (iii) of the latter theorem, we have uniqueness of the decomposition $M\cong M_1\times M_2$.

Now, we prove that if $M$ is an EMV-algebra with a square root $r$ and $N$ is its representing EMV-algebra with top element, then $N$ has a square root $R$ such that $R\big|_{M}=r$.

\begin{thm}\label{3.12}
Let $(M;\vee,\wedge,\oplus,0)$ be an EMV-algebra  and $(N;\vee,\wedge,\oplus,0)$ be its representing EMV-algebra with top element. If $r:M\ra M$ is a square root on $N$, there exists a square root $R:N\to N$ such that $r:=R\big|_{M}$.
\end{thm}

\begin{proof}
If $M$ has a top element, the statement is trivial. So let $M$ have no top element. Without loss of generality, we can assume that $M\subset N$. Since $M$ has a square root, by Theorem \ref{3.11}, there are three cases.

(1) $M$ is a generalized Boolean algebra.
Then $N$ is a Boolean algebra. By Theorem \ref{3.5}, $r(0)=0$, so for each $x\in M$, we have
$r(x)=x\vee r(0)=x$ (by Remark \ref{rmk-cor}). On the other hand, since $N$ is a Boolean algebra, the identity map $R:N\to N$ is a square root.
Hence, $N$ has a square root $R$ and clearly  $R\big|_{M}(x)=x=r(x)$ for all $x\in M$.

(2) $M$ is a strict EMV-algebra. By Corollary \ref{strict-N}, $M$ has a top element and $N=M$. So, the proof for this case is clear.

(3) $M$ is isomorphic to the direct product $M_1\times M_2$, where $M_1$ is a generalized Boolean algebra and $M_2$ is a strict EMV-algebra. By Theorem \ref{3.10}, we know $M_2$ has a top element, so we can easily show that $N$ is isomorphic to $B\times M_2$, where $B$ is a Boolean algebra, and $M_1$ is a maximal ideal of the Boolean algebra $B$.
Since $B$ is a Boolean algebra, the identity map $s_1:B\to B$ is a square root. Let $s_2$ be the square root on the strict EMV-algebra $M_2$.
Easy calculations show that $R:B\times M_2\to B\times M_2$ defined by $R(x,y)=(s_1(x),s_2(x))$ is a square root. Proposition \ref{Homo} implies that
$N$ has a square root $T:N\to N$ too. Now, by Theorem \ref{3.6}, the map $T\big|_{M}M\to M$ is a square root on $M$.
It follows from the note after Definition \ref{3.1} that $T\big|_{M}=r$.

From (1)--(3), we conclude that if $M$ has a square root, so does its representing EMV-algebra with top element $N$.
\end{proof}

We strengthen (iii) of Theorem \ref{3.11}:

\begin{prop}\label{pr:uniq}
Let an EMV-algebra $M$ with square root be isomorphic to the direct product $M_1\times M_2$, where $M_1$ is a generalized Boolean algebra, and $M_2$ is a strict EMV-algebra. In such a case, $M_1$ and $M_2$ are uniquely determined by $M \cong M_1 \times M_2$ up to isomorphism.
\end{prop}

\begin{proof}
Let $M\cong M_1\times M_2\cong M'_1\times M'_2$, where $M_1,M_1'$ are generalized Boolean algebras and $M_2,M_2'$ are strict EMV-algebras and therefore with top elements, see Theorem \ref{3.11}. If $M$ is with top element, $M$ is equivalent to an MV-algebra, and the uniqueness follows from \cite[Thm 2.21]{Hol}.

Assume that $M$ has no top element and let $r$ be a square root on $M$. From Theorem \ref{2.4}, $M$ can be embedded into an EMV-algebra $N$ with top element as a maximal ideal of $N$, and every element $y\in N$ either is in the image of $M$ or is a complement of some element from $M$. Theorem \ref{3.12} asserts that the square root $r$ can be uniquely extended to a square root $R$ on $N$.
The generalized Boolean algebras $M_1$ and $M'_1$ are without top elements and they can be by Theorem \ref{2.4} embedded into Boolean algebras $B_1$ and $B_2$ with top elements, respectively, see also \cite[Thm 2.2]{CoDa}. We claim that $B_1\times M_2$ and $B_2\times M'_2$ are EMV-algebras with top element representing $M_1\times M_2$ and $M_1'\times M_2'$, respectively, because $M_1\times M_2$ is a maximal ideal of $B_1\times M_2$ and every element of $B_1\times M_2$ either belongs to the image of $M_1\times M_2$ or is a complement of some element from the image of $M_1\times M_2$. Similar reasonings also hold for $B_2\times M'_2$. In other words, $B_1\times M_2$ and $B_2\times M'_2$ are EMV-algebras with top element representing $M$ and $N\cong B_1\times M_2\cong B_2\times M'_2$. The uniqueness of the decomposition for $N$, see \cite[Thm 2.21]{Hol}, entails that $B_1\cong B_2$, $M_2\cong M'_2$. We have $M_1\cong (M_1\times M_2)/M_2\cong (M_1'\times M'_2)/M'_2 \cong M_1'$.
\end{proof}

From Theorem \ref{3.12} it can be easily obtained the following corollary.

\begin{cor}\label{3.14}
If $(M;\vee,\wedge,\oplus,0)$ is an EMV-algebra with a square root and $(N;\vee,\wedge,\oplus,0)$ is its representing EMV-algebra with top element, then $N$ satisfies only one of the statements of Theorem \ref{3.11}.
\end{cor}

Finally, the following proposition helps us to provide some examples of EMV-algebras with square roots.

\begin{prop}\label{pr:exist}
Let  $(M;\vee,\wedge,\oplus,0)$ be a locally complete EMV-algebra. Then $M$ has a square root \iff the following statements hold:
\begin{enumerate}[nolistsep]
\item[{\rm (i)}] The set $\{y\in M\colon y\odot y=0\}$ has an upper bound in $M$
\item[{\rm (ii)}] The mapping $\Delta:M\to M$  defined by $\Delta(x):=x^2$, $x\in M$, is onto.
\end{enumerate}
\end{prop}

\begin{proof}
Let $M$ be a locally complete EMV-algebra. Assume (i) and (ii), and let
$u$ be an upper bound for the set $\{y\in M\mid y\odot y=0\}$. Due to (i),
$t:=\bigvee\{y\in M\mid y\odot y=0\}$ exists in $M$.
By Lemma \ref{lem3.2}(ii),
\begin{eqnarray}
\label{Eq4} t\odot t=(\!\!\!\!\!\!\!\bigvee_{\{y\in M\colon y^2=0\}}\!\!\!\!\!\!\! y)\odot
(\!\!\!\!\!\!\!\bigvee_{\{z\in M\colon z^2=0\}}\!\!\!\!\!\!\! z)=\bigvee_{\{y,z\in M\colon y^2=z^2=0\}}\!\!\!\!\!y\odot z\leq
\bigvee_{\{y,z\in M\colon y^2=z^2=0\}}\!\!\!\Big((y^2)\vee (z^2)\Big)=0.
\end{eqnarray}
That is $t=\max\{y\in M\mid y\odot y=0\}$. Set $r(0):=t$.

Given $x\in M$, consider the complete MV-algebra $[0,a]$, where $u,x\leq
a\in\mI(M)$. If $y\in M$ is such that $y\odot y\leq x$, then $(y\ominus
x)\odot (y\ominus x)\leq (y\odot y)\ominus x=0$ (see Remark
\ref{rmk-cor}(ii)) and so
$y\ominus x\leq t\leq u$ which means $y\leq x\oplus t\leq a$.
Hence
\begin{eqnarray}
\label{EQS2}\{y\in M\colon y\odot y\leq x\}=\{y\in [0,a]\colon y\odot y\leq x\}.
\end{eqnarray}
Since $M$ is locally complete, the element $\bigvee\{y\in [0,a]\mid y\odot y\leq x\}$ exists in $[0,a]$ and in $M$ as well, and it is the same in both cases. Set
\begin{eqnarray*}
r(x):=\bigvee\{y\in [0,a]\mid y\odot y\leq x\}=\bigvee\{y\in M\colon y\odot y\leq x\},\quad \mbox{ by (\ref{EQS2}) }.
\end{eqnarray*}
In particular, this yields (Sq2).

Analogously to  (\ref{Eq4}), we can show that $r(x)\odot r(x)\le x$:
Using Lemma \ref{lem3.2}(ii),
we have

\begin{eqnarray*}
r(x)\odot r(x)=\bigvee_{\{y,z\in M\colon y^2,z^2\le x\}}y\odot z\leq
\bigvee_{\{y,z\in M\colon y^2, z^2\le x\}}\Big((y^2)\vee (z^2)\Big)\le x.
\end{eqnarray*}
This gives $r(x)=\max\{y\in M\mid y\odot y \le x\}$.

According to (ii), there is $z\in M$ such that $z\odot z=x$, which
entails $x= z\odot z \le r(x)\odot r(x)\le x$ and it gives (Sq1).
Consequently, $r$ is a square root on $M$.

Conversely, let $r$ be a square root on $M$. Then (i) is
straightforward, and for (ii), we have given $x\in M$, there is $y=r(x)$
such that $y\odot y = x$.
\end{proof}

Let us comment Proposition \ref{pr:exist} with two remarks.

\begin{rmk}\label{rm:contra1}
We show that not every locally complete EMV-algebra has a square root. In other words, there is a locally complete EMV-algebra such that the set $\{y\in M\mid y \odot y=0\}$ has no upper bound:

Given any integer $i\ge 1$, let $M_i$ be the MV-algebra of the real interval $[0,1]$ that is a complete one. Define $M=\sum_iM_i$. According to \cite[Ex 3.2(iii)]{Dvz4}, $M$ is a locally complete EMV-algebra without any top element. Define a countable family $\{y_n=(y^n_i)_i\}_n$ of elements of $M$ such that $y^n_i=0$ if $i>n$ and $y^n_i=1/2$ if $i\le n$. Then $y_n\odot y_n=0:=(0_i)_i$ for each $n$, but the set $\{y_n\mid n\ge 1\}$ has no upper bound in $M$. Consequently, $M$ has no square root, see also Example \ref{3.4}(iv).
\end{rmk}

\begin{rmk}\label{rm:contra2}
It can happen that in a locally complete EMV-algebra $M$, the set $\{y\in M\mid y \odot y=0\}$ has an upper bound, but $M$ has no square roots.

For example, given an integer $n\ge 1$, let us define finite MV-algebras $M_n= \{0,1/n,2/n,\ldots,n/n\}$. Since every $M_n$ is trivially a complete MV-algebra, the set $\{y\in M\mid y \odot y=0\}$ has an upper bound. We assert that $M_n$ has square roots if and only if $n=1$. If $n=1$, then $M_1$ is a Boolean algebra, so the identity function on $M_1$ is a square root function. Thus let $n\ge 2$.

If $r_n$ is a square root on $M$, then it is injective, and the range of $r_n$ is finite and linearly ordered, that is, $r_n=\id_n$, the identity on $M_n$. Consequently, $r_n(0)=0$ and by Theorem \ref{3.5}, $M_n$ is a Boolean algebra, a contradiction.

We note that in $M_n$, the condition (ii) of Proposition \ref{pr:exist} is not satisfied: $0\odot 0=0= \frac{1}{n}\odot\frac{1}{n}$.

In addition, if $M$ is a finite EMV-algebra, it has a top element, and if $M$ is not a Boolean algebra, it does not have any square root.
\end{rmk}

Let $M$ and $E$ be two EMV-algebras and $r:M\to M$ and $s:E\to E$ be square roots. If $f:M\to E$ is a homomorphism of EMV-algebras, then
for each $x\in M$, $f(r(x))\odot f(r(x))=f(x)$, and so by (Sq2), $f(r(x))\leq s(f(x))$. We say $f$ {\em preserves square roots} if
$f(r(x))=s(f(x))$ for all $x\in M$. For example, every homomorphism between generalized Boolean algebras preserves square roots. In the next theorem, we show a necessary and sufficient condition that
a homomorphism of EMV-algebras preserves square roots.

\begin{thm}\label{corHome}
Let $M$ and $E$ be two EMV-algebras with square roots $r:M\to M$ and $s:E\to E$, and $f:M\to E$ be a homomorphism of EMV-algebras.
Then $f$ preserves square roots \iff $\im(f)$ is closed under $s$.

Consequently, every surjective homomorphism of EMV-algebras preserves square roots. On the other side, if $M=\{0\}$ $(M=\{0,1\})$, then the embedding $f$ of $M$ into any EMV-algebra $($any EMV-algebra with top element, $f(1)=1)$ that is not a generalized Boolean algebra is not preserving square roots.
\end{thm}

\begin{proof}
First, we assume that $\im(f)$ is closed under $s$. Then $s\big|_{\im(f)}:\im(f)\to \im(f)$ is a square root on the EMV-algebra $\im(f)$.
Since by Proposition \ref{Homo}, the map $t:\im(f)\to \im(f)$ defined by $t(f(x))=f(r(x))$ is a square root, then $t=s$ which implies that
$s(f(x))=t(f(x))=f(r(x))$ for all $x\in M$.

Conversely, let $f$ preserve square roots. For each $y=f(x)\in\im(f)$, we have $s(y)=f(r(x))\in\im(f)$. Therefore, $\im(f)$ is closed under $s$.
\end{proof}

\begin{cor}
Let $f:M_1\to M_2$ be a homomorphism of EMV-algebras and $N_1$ and $N_2$ be the representing EMV-algebras with top element of $M_1$ and $M_2$, respectively. Consider the homomorphism $\overline{f}:N_1\to N_2$ which is induced from $f$ {\rm(see \cite[Prop 6.1]{Dvz})}.
If $R_1:N_1\to N_1$ and $R_2:N_2\to N_2$ are square roots, then $\overline f$ preserves square roots \iff $f$ has this property.
\end{cor}

\begin{proof}
Due to Theorem \ref{3.6}, we know that $M_1$ and $M_2$ have square roots, say $r_1$ and $r_2$.
Recall that, $\overline{f}(x)=f(x)$, if $x\in M_1$ and $\overline f(x)=f(x')'$, if $x\in N_1\setminus M_1$, where $x'=\lambda_1(x)$ and
$y'=\lambda_1(y)$ for all $x\in N_1$ and $y\in N_2$. Note that $N_1$ and $N_2$ have top elements (denoted by the same $1$).
Let $f$ preserve square roots. Choose $x\in N_1$.

(1) If $x\in M_1$, then by Theorem \ref{3.6}, $R_1(x)=r_1(x)\in M_1$ and clearly $\overline f(R_1(x))=f(r_1(x))=r_2(f(x))=R_2(\overline f(x))$.

(2) If $x\in N_1\setminus M_1$, then $x'\in M_1$ and $\overline f(R_1(x'))=R_2(f(x'))$ (by (1)). Then
\begin{eqnarray*}
R_2(\overline f(x))&=& R_2(f(x')')=R_2(f(x')\to 0)=R_2(f(x'))\to R_2(0),\mbox { by Proposition \ref{3.2}(x)}\\
&=& \overline f(R_1(x'))\to R_2(0)=\overline f(R_1(x'))\to R_2(\overline f(0))=\overline f(R_1(x'))\to \overline f(R_1(0))\\
&=& \overline f\Big( R_1(x')\to R_1(0) \Big)=\overline f( R_1(x'\to 0)),\mbox { by Proposition \ref{3.2}(x)}\\
&=& \overline f(R_1(x)).
\end{eqnarray*}
Hence, $\overline f$ preserves square roots. The proof of the converse is clear, since for each $x\in M_1$,
$\overline f(x)=f(x)\in M_2$, $R_1(x)=r_1(x)$ and $R_2(f(x))=r_2(f(x))$.
\end{proof}

\begin{prop}\label{pr:subalg}
Let $M$ be an infinite MV-subalgebra of the MV-algebra of the real interval $[0,1]$. Then $M$ has square roots if and only if, $x\in M$ implies $(x+1)/2\in M$.

If $M$ has a square root, $r$, the square root is the restriction of the square root on $[0,1]$, i.e. $r(x)=(x+1)/2$, $x\in M$, so that $M$ is strict, and the MV-embedding of $M$ into $[0,1]$ preserves square roots.
\end{prop}

\begin{proof}
We recall that every subgroup of $[0,1]$ containing $1$ is either of the form $\frac{1}{n}\mathbb Z$, or is dense in $\mathbb R$, see e.g. \cite[Lem 4.21]{Goo}, in our case, $M$ is dense in $[0,1]$.

Assume that $r$ is a square root on $M$. Due to Example \ref{3.4}(ii), the MV-algebra $[0,1]$ has the square root $s(x)=(x+1)/2$, $x\in [0,1]$. Then for each $x\in M$, $r(x)\le s(x)$. Let $y\in [0,1]$ be such $y\odot y\le x$. Since $M$ is dense in $[0,1]$, there is a sequence $(y_n)_n$ of elements of $M$ such that $(y_n)_n\nearrow y$ ($(y_n)_n$ is non-decreasing and converges to $y$).
Then $y = \lim_n y_n\le r(x)$ for each $y\in [0,1]$ with $y\odot y\le x$, which yields $s(x)\le r(x)$, i.e. $r$ is the restriction of $s$ and $(x+1)/2\in M$ for each $x\in M$. Moreover, the embedding $M$ into $[0,1]$ preserves square roots.

The converse statement is evident.
\end{proof}

\begin{rmk}
(1) Due to Example \ref{3.4}, the MV-algebra of rational numbers and the MV-algebra of dyadic numbers satisfy the condition of the latter proposition, and they are strict.

(2) According to Remark \ref{rm:contra2}, the MV-algebra $\{0,1/n,2/n,\ldots,n/n\}$, $n\ge 1$, has square roots if and only if $n=1$.

(3) An arbitrary MV-subalgebra of $[0,1]$ (not only infinite) has square roots iff for each $x\in M$, $(x+1)/2\in M$.

(4) If $M$ is an MV-algebra with a square root $r$ and $f:M\to [0,1]$ is an MV-homomorphism, then $f$ preserves square roots iff $f(r(x))=(f(x)+1)/2$, $x\in M$.
\end{rmk}

\begin{exm}\label{ex:irr}
(1) For each irrational $\alpha$, $0<\alpha <1/2$, let $M(\alpha)=
\{m+n\alpha \mid m,n \in \mathbb Z, 0\le m+n\alpha\le 1\}$. Due to an example just after \cite[Cor 7.2.6]{CDM}, $M(\alpha)$ is an MV-subalgebra of $[0,1]$ generated by $\alpha$. If $\beta$ is an irrational in $[0,1/2]$, $M(\alpha)\cong M(\beta)$ iff $M(\alpha)=M(\beta)$ iff $\alpha=\beta$.

We assert that $M(\alpha)$ has no square root. Indeed, otherwise, if $r$ is its square root, $r(\alpha)=(\alpha+1)/2\in M(\alpha)$ which implies $(\alpha +1)/2=m+n\alpha $ for some $m,n\in \mathbb Z$, giving $\alpha$ is rational, contradiction. Therefore, $M$ has no square roots, see Proposition \ref{pr:subalg}.

Consequently, there are uncountably many MV-subalgebras of $[0,1]$ having no square roots.

(2) For any integer $p$, let $M(p)=\{i/p^n\mid i=0,1,\ldots,p^n, n\ge 1\}$, the set of $p$-adic numbers in $[0,1]$. If $p$ is a prime number, $p\ge 3$, then $M(p)$ is an MV-algebra with no square roots. Indeed, test the criterion from Proposition \ref{pr:subalg}: Let $x=2/p$ and assume $(2/p+1)/2 = (2+p)/2p= i/p^{n+1}$ for some $i=0,1,\ldots, p^{n+1}$ and $n\ge 0$. It entails $(2+p)p^n=2i$, which is a contradiction while on the left-hand side, we have an odd number, whereas, on the right-hand side, it is an even number. Therefore, $(x+1)/2\notin M(p)$.

The same trick shows that $M(p)$ has no square root if $p$ is an odd number.

We note that if $p=1$, $M(1)=\{0,1\}$, so it has a square root, and if $p=2$, $M(p)$ are dyadic numbers in $[0,1]$ and it also has a square root, see Example \ref{3.4}(viii).

(3) On the other hand, if $p\ge 3$ is an even number, then $M(p)$ has a strict square root. Put $x=j/p^k$, where $j=0,1,\ldots, p^k$. We search for integers $m$ and $i=0,1,\ldots,p^m$ such that $(x+1)/2= (j/p^k+1)/2= (j+p^k)/2p^k = i/p^m$. Without loss of generality, we can assume $m> k$, i.e. e.g. $m=k+n$, where $n>0$. Then
\begin{equation}\label{eq:i}
(j+p^k)/2p^k=i/p^{k+n}
\end{equation}
for $i=0,1,\ldots, p^{k+n}$, which gives $(j+p^k)p^n/2= i$. Since $2$ divides $p$, $i$ is an integer. Moreover, $i= (j+p^k)p^n/2\le (p^k+p^k)p^n/2 = p^{n+k}$, so that the right-hand side of \eqref{eq:i} has a solution for $i$, i.e. $(x+1)/2\in M(p)$. If $s$ is a square root on $M(p)$, then $s(x)=(x+1)/2$, $x\in M(p)$, so that $M(p)$ is strict, see Proposition \ref{pr:subalg}. 

We show that there is countably many mutually different unital subalgebras $(M(p),1)$ with the square root. Let $2=p_1<3=p_2<\cdots<p_n$ be the first $n$ prime numbers and let $P_n$ be its product. We assert for each $n\ge 1$, the number $1/p_{n+1}\in M(P_{n+1})\setminus M(P_n)$. We have $1/p_{n+1} = (p_1\cdots p_n)/P_{n+1}\in M(P_{n+1})$. On the other side, if $1/p_{n+1}=i/P_n^k$, then $P_n^k=ip_{n+1}$ and this equation has no solution $i$ in integers, so $1/p_{n+1}\not\in M(P_n)$.
\end{exm}

\section{Divisible EMV-algebras and Complete Description of Square Roots on MV-algebras}

In the sequel, we find relations between strict EMV-algebras and some other subclasses of EMV-algebras such as divisible and locally complete EMV-algebras and present some examples of divisible EMV-algebras with square roots. Moreover, we present square roots on tribes, EMV-tribes, and we present a complete characterization of any square root on an MV-algebra by group addition in the corresponding unital $\ell$-group.

\begin{defn}\label{4.1}
An EMV-algebra $(M;\vee,\wedge,\oplus,0)$ is called {\em divisible} if, for each $x\in M$ and each $n\in\mathbb N$, there exists $y\in M$
such that $n.y=x$ and $(n-1).y\odot y=0$.
\end{defn}

For example, the MV-algebra $[0,1]$ and the MV-algebra of rational numbers in $[0,1]$ are divisible EMV-algebras whereas the MV-algebra of dyadic numbers not. We note that an $\ell$-group is {\it divisible} if for each $x\in G$ and each integer $n\ge 1$, there is $y\in G$ such that $ny=x$. It is easy to show that if $M=\Gamma(G,u)$, then $M$ is divisible iff $G$ is divisible, see e.g. \cite[Lem 2.3]{DiSe}.

\begin{prop}\label{pr:embed}
Every MV-algebra can be embedded into a divisible MV-algebra with a strict square root.
\end{prop}

\begin{proof}
It is well-known that every Abelian $\ell$-group can be embedded into a divisible Abelian $\ell$-group, see e.g. \cite[Page 4]{Gla}. Therefore, every MV-algebra $M=\Gamma(G,u)$ can be embedded into an MV-algebra with square root. Indeed, take a divisible hull $G^d$ of $G$, $u$ is also a strong unit of $G^d$, $M$ can be embedded into $\Gamma(G^d,u)$, and $s(x)=(x+u)/2$, $x\in\Gamma(G^d,u)$, is a strict square root on $\Gamma(G^d,u)$.
\end{proof}

\begin{rmk}\label{4.2}
(i) Clearly, by \cite{LaLe}, any divisible MV-algebra is a divisible EMV-algebra with top element.

(ii) If $(M;\vee,\wedge,\oplus,0)$ is a divisible EMV-algebra, then by Lemma \ref{2.3}, for each idempotent element $a\in\mI(M)$, the MV-algebra $([0,a];\oplus,\lam_a,0,a)$ is a divisible MV-algebra. Easy calculations show that the converse also holds.

(iii) Consider a non-finite family $\{(M_i;\vee_i,\wedge_i,\oplus_i,0_i)\mid i\in I\}$ of divisible MV-algebras. Let
$M:=\sum_{i\in I}M_i$ (see \cite{Dvz}) which is a proper EMV-algebra. Then $M$ is a divisible EMV-algebra. Indeed,
let $x\in\sum_{i\in I} M_i$ and $n\in\mathbb N$. Then $x\in\prod_{i\in I} M_i$ is with finite support.
Let $\supp(x)=\{i_1,\ldots, i_n\}$. Without loss of generality, we can assume that $x=(x_i)_{i\in I}$, where
$x_i=0_i$ for each $i\in I\setminus\supp(f)$. By the assumption, for each $j\in \supp(x)$, there is $y_j\in M_j$ such that
$x_j=n.y_j$ and $(n-1).y_j\odot_i y_j=0_j$. Set $z:=(z_i)_{i\in I}$ where $z_i=0_i$ for $i\in I\setminus\supp(x)$ and $z_i=y_i$ for all
$i\in \supp(x)$. Then clearly $z\in \sum_{i\in I} M_i$, $n.z=x$ and $(n-1).z\odot z=0$.
\end{rmk}

Recall that if $(M;\oplus,',0,1)$ is a divisible MV-algebra, then for each $x\in M$ and $n\in \mathbb N$, there exists $y\in M$ such that
$x'=n.y$ and $(n-1).y\odot y=0$ which imply that $x=(n.y)'=(y')^n$ and $(y')^{n-1}\oplus y'=1$. This note will be used in the next proposition.

\begin{prop}\label{4.4}
Let $(M;\vee,\wedge,\oplus,0)$ be an EMV-algebra and $(N;\vee,\wedge,\oplus,0)$ be its representing EMV-algebra with top element. If $M$ is divisible, then so is $N$.
\end{prop}

\begin{proof}
Let $M$ be divisible. It suffices to show that, for each $x\in N\setminus M$ and each $n\in\mathbb N$, there exists $y\in N$ such that
$x=n.y$ and $(n-1).y\odot y=0$. Choose $x\in N\setminus M$ and $n\in\mathbb N$. Then $x'\in M$.
Choose $a\in\mI(M)$ such that $x\leq a$. Consider the MV-algebra $([0,a];\oplus,\lam_a,0,a)$ which is divisible by Remark \ref{4.2}(ii). By the note just before the proposition, there is $y\in [0,a]$ with $x'=y^n$ and $(\lam_a(y))^{n-1}\oplus \lam_a(y)=a$.
It follows that $x=(y^n)'=n.y'$ and $0=\lam_a(a)=\lam_a\left(\left(\lam_a(y)\right)^{n-1}\oplus \lam_a(y)\right)=(n-1).y\odot y$. Therefore, $N$ is divisible.
\end{proof}

Recall that an MV-algebra $M$ is said to be {\em injective} if it is an injective object in the category of MV-algebras which means that given MV-algebras $X$ and $Y$, for each one-to-one homomorphism $f:X\to Y$ and each homomorphism $g:X\to M$, there is a homomorphism $h:Y\to M$ such that $h\circ f=g$. By \cite[page 17]{DiSe}, we know that $M$ is injective \iff it is complete and divisible.

Analogously we also define an {\it injective EMV-algebra}.

Let $(M;\vee,\wedge,\oplus,0)$ be a locally complete and divisible EMV-algebra. For each $a\in\mI(M)$, the MV-algebra $([0,a];\oplus,\lam_a,0,a)$ is complete and divisible, and Remark \ref{4.2}(ii) entails that $[0,a]$ is an injective MV-algebra (see \cite[page 17]{DiSe}).

\begin{defn}\label{4.5}
An EMV-algebra $(M;\vee,\wedge,\oplus,0)$ is said to be {\em strongly atomless} if, for each $x\in M$, there exists $a\in\mI(M)$ such that
$x\leq a$ and $[0,a]$ is a strongly atomless MV-algebra.
\end{defn}

\begin{prop}\label{4.6}
Let $(M;\vee,\wedge,\oplus,0)$ be an EMV-algebra and $N$ be its representing EMV-algebra with top element.
\begin{itemize}[nolistsep]
\item[{\rm (i)}] The EMV-algebra $M$ is strongly atomless if and only if, for each $a\in\mI(M)$, the MV-algebra $[0,a]$ is a strongly atomless MV-algebra.
\item[{\rm (ii)}] If $M$ is locally complete, then $M$ is strongly atomless \iff it is divisible.
\item[{\rm (iii)}] If $N$ is strongly atomless, then $N=M$. Moreover, if $M$ is strongly atomless, then $N$ is not necessarily strongly atomless.
\end{itemize}
\end{prop}

\begin{proof}
(i) Let $a\in\mI(M)$ and $0<x\leq a$. By the assumptions, there exists $b\in \mI(M)$ such that $x\leq b$ and $[0,b]$ is a strongly atomless MV-algebra.
Hence, there is $z\in [0,b]$ with $0<z<x$ and $x\odot\lam_b(z)\leq z$. Then $z\in [0,a]$ and by Proposition \ref{2.3}, we get
$x\odot\lam_a(z)=x\odot (\lam_b(z)\wedge a)=(x\odot\lam_b(z))\wedge (x\odot a)\leq z$.

(ii) Let $M$ be locally complete and divisible. For each $a\in\mI(M)$, the MV-algebra $[0,a]$ is complete and divisible.
So, \cite[Thm 6.17]{Hol} implies that $[0,a]$ is strongly atomless. By (i), $M$ is strongly atomless.

(iii) First, let $N$ be strongly atomless. From \cite[Thm 6.17]{Hol} it follows that $N$ is strict. Take an idempotent element $a\in \mathcal I(M)$ such that $s(0)\leq a$. By Corollary \ref{strict-N}, $1\in M$ (since $\mI(M)$ is a full subset of $M$),
which means $M=N$ (note that $M$ is an ideal of $N$).

Now, consider a non-finite family $\{M_i\mid i\in I\}$ of complete and divisible MV-algebras. Clearly, $M:=\sum_{i\in I}M_i$ is a
locally complete and divisible EMV-algebra that does not have a top element. By (ii), $M$ is strongly atomless, and its representing EMV-algebra with top element is not strongly atomless (by the first part).
\end{proof}

\begin{thm}\label{4.7}
Let $r$ be a square root on a locally complete EMV-algebra $(M;\vee,\wedge,\oplus,0)$.
\begin{itemize}[nolistsep]
\item[{\rm (i)}]  $M$ is divisible \iff it has a strict square root on $M$.
\item[{\rm (ii)}] Any locally complete divisible EMV-algebra with a square root has a top element.
\end{itemize}
\end{thm}

\begin{proof}
(i) Assume that $M$ is locally complete and divisible. For each $a\in \mI(M)$, the MV-algebra $[0,a]$ is complete and divisible.
By \cite[Thm 6.17]{Hol}, for each $a\in\mI(M)$, the square root $r_a:[0,a]\to [0,a]$ is strict. So, by definition, $r$ is a strict square root.
Conversely, let $M$ be a locally complete and strict EMV-algebra. For each $a\in \mI(M)$, the MV-algebra $[0,a]$ is complete and strict, which implies that it is divisible  \cite[Thm 6.17]{Hol}. It follows from Remark \ref{4.2}(ii) that $M$ is divisible.

(ii) The proof follows from (i) and Theorem \ref{3.10}.
\end{proof}

\begin{thm}\label{injemv}
Every injective EMV-algebra is locally complete, divisible, and with top element.
\end{thm}

\begin{proof}
Let $Q$ be an injective EMV-algebra. Put $a\in \mI(Q)$.
First, we show that $[0,a]$ is an injective MV-algebra. Let $X$ and $Y$ be MV-algebras, $f:X\to Y$ be a one-to-one MV-homomorphism and
$g:X\to [0,a]$ be an MV-homomorphism (see Figure \ref{fig1}).
     \setlength{\unitlength}{1mm}
\begin{figure}[!ht]
    \begin{center}
        \begin{picture}(40,20)
        \put(12,17){\vector(2,0){12}}
        \put(8,14){\vector(0,-1){12}}
        \put(5,-2){\makebox(4,2){{ $[0,a]$}}}
        \put(5,16){\makebox(4,2){{ $X$}}}
        \put(27,16){\makebox(4,2){{ Y}}}
        \put(15,19){\makebox(4,2){{ $f$}}}
        \put(1,7){\makebox(4,2){{$g$}}}
        \end{picture}
        \caption{\label{fig1} Injective object}
    \end{center}
\end{figure}
Consider the inclusion map $i:[0,a]\to Q$ which is an EMV-homomorphism. By the assumption, there exists $h:Y\to Q$ such that $h\circ f=g$.
Denote the top elements of $X$ and $Y$ by the same notation, $1$. Since $a=g(1)=h(f(1))=h(1)$ and $h(y)\leq h(1)$ for all $x\in Y$, then we have
$\im(h)\s [0,a]$. Moreover, by the definition of an EMV-homomorphism, $h:[0,1]\to [0,h(1)]=[0,a]$ is an MV-homomorphism. That is, $[0,a]$ is an injective MV-algebra. Now, by \cite[Thm 2.14]{DiSe}, $[0,a]$ is a complete and divisible MV-algebra.
Then clearly, $Q$ is a divisible and locally complete EMV-algebra.

Theorem \ref{4.7}(ii) entails that any injective EMV-algebra with square root has a top element.
\end{proof}

Now, we present some EMV-algebras of fuzzy sets and we show when they have square roots. As it was said, every EMV-algebra with top element is equivalent to an MV-algebra, and in the rest of the paper, we will concentrate mainly on square roots on MV-algebras.

A {\it tribe} is a system $\mathcal F$ of fuzzy sets of a non-empty set $\Omega$ such that (i) $0_\Omega\in \mathcal F$, (ii) $f\in \mathcal F$ implies $1-f \in \mathcal F$, and (iii) if $f_n\in \mathcal F$, $n\ge 1$, then $\bigoplus_{n=1}^\infty f_n:= \min\{\sum_{n=1}^\infty f_n, 1\}\in \mathcal F$. It is a $\sigma$-complete MV-algebra where all MV-algebraic operations are defined by points.
Any tribe is a generalization of a $\sigma$-algebra $\mathcal S$ of subsets: Indeed, if $A_n\in \mathcal S$, $n\ge 1$, then $\chi_{\bigcup_n A_n}= \bigoplus_n \chi_{A_n}$. Due to \cite{Dvu, Mun}, every $\sigma$-complete MV-algebra is a $\sigma$-homomorphic image of some tribe. Clearly, every $\{0,1/n,\ldots,n/n\}$ is a finite tribe.

Given a tribe $\mathcal F$, let $\mathcal S=\mathcal S(\mathcal F)=\{A\subseteq \Omega \mid \chi_A \in \mathcal F\}$. Due to \cite[Thm 8.1.4]{RiNe}, $\mathcal S$ is a $\sigma$-algebra of subsets of $\Omega$.

\begin{prop}\label{pr:tribe}
Let $\mathcal F$ be a tribe of fuzzy sets of a set $\Omega\ne \emptyset$. The following statements are equivalent:
\begin{itemize}[nolistsep]
\item[{\rm (i)}] Given $f\in \mathcal F$, $(f+1)/2$ belongs to $\mathcal F$.
\item[{\rm (ii)}] The MV-algebra of dyadic numbers can be embedded into $\mathcal F$.
\item[{\rm (iii)}] The tribe $\mathcal F$ contains all constant functions on $\Omega$.
\item[{\rm (iv)}] The tribe $\mathcal F$ contains all $\mathcal S$-measurable functions on $\Omega$.
\item[{\rm (v)}] The tribe $\mathcal F$ is divisible.
\end{itemize}
In either case, $\mathcal F$ has a square root $s(f)=(f+1)/2$, $f\in \mathcal F$, and $\mathcal F$ is strict.
\end{prop}

\begin{proof}
(i) $\Rightarrow$ (ii). Set $s(f)=(f+1)/2$ for each $f\in \mathcal F$. If $f=0$, then $1/2\in \mathcal F$.  Assume by induction that $1/2^n\in \mathcal F$, then $i/2^n\in \mathcal F$ for each $i=0,1,\ldots,2^n$. Moreover, $(i/2^n +1)/2= (i+2^n)/2^{n+1}\in \mathcal F$, so that $s(2/2^n)-s(1/2^n)= 1/2^{n+1}\in \mathcal F$ and $j/2^{n+1}\in \mathcal F$ for each $j=0,1,\ldots,2^{n+1}$. That is, $\mathcal F$ contains all dyadic constants in $\mathcal F$.

(ii) $\Leftrightarrow$ (iii). Since $\mathcal F$ is $\sigma$-complete, $\mathcal F$ contains all constant functions on $\Omega$. The converse implication is clear.

(iii) $\Leftrightarrow$ (iv). It follows from \cite[Thm 8.1.4]{RiNe}.
It is easy to see (i) implies $s$ defined by $s(f)=(f+1)/2$, $f\in \mathcal F$, is a square root on $\mathcal F$.

It is clear that (v) is equivalent with (i)--(iv).

In either case, $s(0)=1/2= (1/2)'$, so that $\mathcal F$ is strict.
\end{proof}

We note that if a tribe $\mathcal F$ has a square root $s$, then not necessarily $s(f)=(f+1)/2$, $f \in \mathcal F$. Indeed, let $\mathcal F$ be the system of characteristic functions of a $\sigma$-algebra $\mathcal S$. Then $\mathcal S$ is a Boolean algebra and the identity function $\id$ is a unique square root on $\mathcal F$. Of course, $0\ne (0+1)/2$.

A Riesz MV-algebra is an algebra $(M;\oplus,',0,1, \{\alpha\}_{\alpha\in [0,1]})$, where $(M;\oplus,',0,1)$ is an MV-algebra and $\alpha$ is a unary operation on $M$ such that $\alpha(x\oplus y)=(\alpha x)\oplus (\alpha y)$ and $(\alpha(\beta(x))=(\alpha\beta)(x)$ for all $x,y \in M$ and $\alpha,\beta \in [0,1]$.

\begin{cor}\label{co:tribe}
Let $M$ be a $\sigma$-complete MV-algebra that is a $\sigma$-homomorphic image of a tribe $\mathcal F$ with a square root $s(f)=(f+1)/2$, $f\in \mathcal F$. Then $M$ has a square root, $M$ is divisible and a Riesz MV-algebra. Moreover, $M$ is strict.
\end{cor}

\begin{proof}
Let $M$ be a $\sigma$-homomorphic image of a tribe $\mathcal F$ with a square root $s(f)=(f+1)/2$, $f \in \mathcal F$, and let $\phi:\mathcal F\to M$ be a surjective $\sigma$-homomorphism. Then $s_\phi(x)=s(\phi(f))$, $x\in M$, where $f\in \mathcal F$ with $\phi(f)=x$, is a square root on $M$, see e.g. Proposition \ref{Homo}. By Proposition \ref{pr:tribe}, $\mathcal F$ consists of all $\mathcal S$-measurable fuzzy sets on $\Omega$, where $\mathcal S=\mathcal S(\mathcal F)$. Therefore, $\mathcal F$ is divisible, and given $x\in M$ and $n\ge 1$, there is an element $f\in \mathcal F$ with $\phi(f)=x$ and $\frac{1}{n}f \in \mathcal F$ and giving $\frac{1}{n}x=\frac{1}{n}\phi(f)=f(\frac{1}{n}f)$. Moreover, $\frac{1}{n}x$ is a unique element $y$ of $M$ such that $n.y= x$ and $(n-1).x\odot x=0$. This yields, $\frac{m}{n}x\in M$ for each $m=0,1,\ldots,n$, $n\ge 1$. Since $\mathcal F$ and $M$ are $\sigma$-complete, for each $\alpha\in [0,1]$, there is a sequence $(\alpha_n)_n$ of rational numbers in $[0,1]$ such that $\alpha_n \nearrow \alpha$. Whence $\alpha x$ is defined in $M$. Since $\mathcal F$ is a Riesz MV-algebra, so is $M$.

Clearly, $s_\phi$ is a strict square root on $M$.
\end{proof}

An MV-algebra $M$ satisfies the {\it two-divisible property} or $M$ is {\it two-divisible}, if given $x\in M$, there is $y\in M$ such that $2.y=x$ and $y\odot y=0$. For example, the MV-algebra of dyadic numbers has the two-divisibility property but not the divisibility property. If $M=\Gamma(G,u)$, then $M$ has the two-divisibility property iff $G$ is two-divisible. We note that a group $G$ is {\it two-divisible} if given $g\in G$, there is $h\in G$ such that $y+y=g$. Since $G$ is an Abelian $\ell$-group, $y$ is unique. For example every Riesz MV-algebra is two-divisible.

\begin{prop}\label{pr:2-div} If $M$ is a two-divisible MV-algebra,
then $M$ is with square root.
\end{prop}

\begin{proof}
Define $s:M\to M$ by $s(x)=(x+u)/2$, $x \in M$. Here $+$ denotes the group addition in $(G,u)$ with $M=\Gamma(G,u)$.
The element $s(x)$ belongs to $M$. We have $s(x)\odot s(x)= (2((x+u)/2) -u)\vee 0= x$. On the other hand, if $y\odot y \le x$, then $(2y-u)\vee 0\le x$ which gives $ y \le y\vee (u/2)= ((2y -u)\vee 0) +u)/2\le (x+u)/2=s(x)$.
\end{proof}

An EMV-generalization of a tribe is an EMV-tribe introduced in \cite{lomi}. A system $\mathcal T\subseteq [0,1]^\Omega$ of fuzzy sets of a set $\Omega\ne \emptyset$ is said to be an {\it EMV-tribe} if
\vspace{1mm}
\begin{enumerate}[nolistsep]
\item[(i)] $0_\Omega \in \mathcal T$ where $0_\Omega(\omega)=0$ for each $\omega \in \Omega$;
\item[(ii)] if $a \in \mathcal T$ is a characteristic function, then (a) if $f\in \mathcal T$ and $f(\omega)\le a(\omega)$ for each $\omega \in \Omega$, then
$a-f \in \mathcal T$
(b) if $\{f_n\}$ is a sequence of functions from $\mathcal T$ with $f_n(\omega)\le a(\omega)$ for each $\omega \in \Omega$ and each $n\ge 1$, where $a\in \mathcal T$ is a characteristic function, then $\bigoplus_n f_n\in \mathcal T$, where $\bigoplus_nf_n(\omega) = \min\{\sum_n f_n(\omega),a(\omega)\}$, $\omega \in \Omega$;
\item[(iii)] for each $f,g \in \mathcal T$, there is a characteristic function $a \in \mathcal T$ such that $f(\omega), g(\omega)\le a(\omega)$ for each $\omega \in \Omega$;
\item[(iv)] given $\omega \in \Omega$, there is $f \in \mathcal T$ such that $f(\omega)=1$.
\end{enumerate}

Then the expression $\bigoplus_nf_n$ does not depend on a characteristic function $a\ge f_n$, $n\ge 1$, and every EMV-tribe of fuzzy sets is a Dedekind $\sigma$-complete EMV-algebra where points define all EMV-operations. Due to \cite{lomi}, every Dedekind $\sigma$-complete EMV-algebra is a $\sigma$-surjective image of some EMV-tribe.

\begin{prop}\label{pr:F-tribe}
Let $\mathcal T$ be an EMV-tribe such that $(f+1)/2\in \mathcal T$ for each $f \in \mathcal T$. Then $s(f)=(f+1)/2$, $f \in \mathcal T$, is a square root on $\mathcal T$. Moreover, $\mathcal T$ is in fact a divisible and strict tribe.
\end{prop}

\begin{proof}
Let $f,g \in \mathcal T$ be such that $f\odot f \le g$. There is a characteristic function $a\in \mathcal T$ such that $f\le a$, so that $f\odot f = (2f-a)\vee 0 \le g$ which gives $f\le (g+a)/2 \le (g+1)/2$. If the condition holds on $\mathcal T$, then it is simply to verify that $s$ is a square root on $\mathcal T$.

Then $s(0)=1/2\in \mathcal T$. As in implication (i) $\Rightarrow$ (ii) of Proposition \ref{pr:tribe}, we can show that $1/2^n\in \mathcal T$ for each $n\ge 1$. The representing EMV-algebra $N(\mathcal T)=\mathcal T \cup \{1-f\mid f\in \mathcal T\}$.
In particular, we have $1=1/2\oplus 1/2 \in \mathcal T$, so that $\mathcal T$ is a strict tribe. By Proposition \ref{pr:tribe}, it is clear that $\mathcal T$ is divisible.
\end{proof}

Now, we present a general form of square roots on MV-algebras.

We note that if $M=\Gamma(G,u)$, then $M$ is totally ordered iff so is the group $G$. We remind that an $\ell$-group $G$ enjoys unique extraction of roots if for all integers $n\ge 1$ and $g,h \in G$, $ng=nh$ implies $g=h$. Due to \cite[Lem 2.1.4]{Gla}, every totally ordered group enjoys unique extraction of roots, and every equation $ny=x$, $x,y \in G$, has a unique solution $y$ denoted as $y=\frac{1}{n}x$. Consequently, every Abelian $\ell$-group
enjoys unique extraction of roots.
\begin{prop}\label{pr:4.1}
Let $(G,u)$ be a totally ordered unital $\ell$-group. If $r$ is a square root on $M=\Gamma(G,u)$, then either $r=\id_M$ if $r(0)=0$ or each element $(x+u)/2$, $x\in M$, exists in $M$, and
$$
r(x)=(x+u)/2,\quad x\in M,
$$ if $r(0)>0$, where $+$ is the group addition in the group $G$. In the second case, $r$ is strict.
\end{prop}

\begin{proof}
If $r(0)=0$, then $M$ is a two-element Boolean algebra $\{0,1\}$, and $r=\id_M$, see \cite[Prop 2.19]{Hol} or Proposition \ref{3.2}(xi).

Assume $r(0)>0$. For any $x\in M$, we have two cases.

(1) Let $x>0$. Then $x=r(x)\odot r(x)=(r(x)+r(x)-u)\vee 0$. Since $M$ is totally ordered, $x=2r(x)-u$ and so $x+u=2r(x)$.
The unique extraction of roots implies that the element $(x+u)/2$ exists in $G$ and $r(x)=(x+u)/2$ belongs to $M$.

(2) Let $x=0$. Then $0=r(0)\odot r(0)=(2r(0)-u)\vee 0$ entails that $2r(0)-u\leq 0$, that is $2r(0)\leq u$. On the other hand, let
$y\in M$. We have $y\odot y\leq 0$ \iff $2y\leq u$, so by (Sq2), $2y\leq u$ implies that $y\leq r(0)$. That is,
$r(0)=\max\{z\in M\mid 2z\leq u\}$. Now, if $g\in G$ be such that $2g\leq u$, then $g\in G^-$ implies that $g\leq r(0)$ and
from $g\in G^+$ we get $g\leq 2g\leq u$, that is $g\in [0,u]=M$ which means $g\leq r(0)$.
Hence, $r(0)=\max\{z\in G\mid 2z\leq u\}$.
On the other hand, due to Proposition \ref{3.2}(xiii), the element $r(0)^-\odot r(0)^-$ is an idempotent element of $M$, so $r(0)^-\odot r(0)^-=0$ or $r(0)^-\odot r(0)^-=1$ (since $M$ is totally ordered).
If $r(0)^-\odot r(0)^-=1$, then $r(0)^-=1$ entails that $r(0)=0$ which is excluded. If $r(0)^-\odot r(0)^-=0$, then $r(0)\oplus r(0)=u$ entails
$r(0)+r(0)\geq u$. Thus, $2r(0)=u$ and $r(0)=(0+u)/2$ which implies $r$ is strict.
\end{proof}

We note, that if $(G,u)$ is a totally ordered unital group such that $r(x):=(x+u)/2$, $x\in M$, exists in $G$, then it exists in $M=\Gamma(G,u)$ and $r$ is a strict square root on $M$ as it is straightforward to verify.

The latter result can be generalized as follows.

\begin{thm}\label{4.3}
Let $(G,u)$ be a unital $\ell$-group and let $M=\Gamma(G,u)$ have a strict square root $r$.
For each $x\in M$, the element $(x+u)/2$ is defined in $G$, and
\begin{equation}\label{eq:equat}
r(x)=(x+u)/2,\quad x\in M,
\end{equation}
where $+$ denotes the group addition in the group $G$.
\end{thm}

\begin{proof}
Set $X:=\spec(M)$. Consider the embedding $\varphi:M\to M_0:=\prod_{P\in X}M/P$ defined by $\varphi(x)=(x/P)_{P\in X}$, $x \in M$, and let $M_0=\Gamma(G_0,u_0)$, where $u_0=\varphi(u)$ and $M_1:=\varphi(M)=\Gamma(G_1,u_1)$, where $G_1\subseteq G_0$ and $u_1=u_0$.

(1) For each $P\in X$, we have $r(0)\notin P$ (otherwise $1=u=r(0)\oplus r(0)\in P$) consequently, $r(0)/P\neq 0/P$.

(2) By \cite[Thm 3.1]{Amb}, the surjective homomorphism $\pi_P\circ\varphi:M\to M/P$ induces a square root $t_P:M/P\to M/P$ defined by
$t_P(x/P)=r(x)/P$, ($x\in M$). Since $M/P$ is totally ordered, there is a totally ordered unital $\ell$-group $(G_P,u_P)$ such that $\Gamma(G_P,u_P)=M/P$. Part (1) implies that $M/P$ is not a Boolean algebra and so by Proposition \ref{pr:4.1},  $r(x)/P=t_P(x/P)=(x/P+_P u/P)/2$ $(x\in M)$, where $+_P$ is the group addition in the unital $\ell$-group $G_P$.

(3) On the MV-algebra $M_0$, we can define a square root $t:M_0\to M_0$ by $t((y_P)_{P\in X})=(t_P(y_P))_{P\in X}$ for each $y=(y_P)_{P\in X}\in M_0$.  We have
$$
\varphi(r(x))=(r(x)/P)_{P\in X}=(t_P(x/P))_{P\in X}=((x/P+_P u/P)/2)_{P\in X} =t(\varphi(x)), \quad x\in M.
$$
Due to the categorical equivalence between the category of unital Abelian $\ell$-groups and the category of MV-algebras, \cite{Mun,CDM}, the MV-injection $\varphi:M\to M_0$ can be uniquely extended to an injective homomorphism of unital $\ell$-groups $\hat \varphi: (G,u)\to (G_0,u_0)$, see e.g. \cite[Lem 7.2.1]{CDM}.

Therefore,
$((x/P+_P u/P)/2)_{P\in X}+((x/P+_P u/P)/2)_{P\in X}=(x/P+_P u/P)_{P\in X}$, so $((x/P+_Pu/P)_{P\in X})/2$ exists in $G_1\subseteq G_0$ and is equal to
$((x/P+_P u/P)/2)_{P\in X}$.

(4) Given $P\in X$, let $\hat P$ be the $\ell$-ideal of $G$ generated by $P$. Then $x/P=x/\hat P$ for each $x\in M$.
By (3), $\varphi(r(x))+_{G_1}\varphi(r(x))=(x/P+_P u/P)_{P\in X}= (x/\hat P +_P u/\hat P)_{P\in X}=((x+u)/\hat P)_{P\in X}\in G_1$, so

\begin{eqnarray*}
r(x)+r(x)&=& \hat\varphi^{-1}(\varphi(r(x))+_{G_1} \varphi(r(x))
= \hat\varphi^{-1}(\hat \varphi(r(x))+_{G_1} \hat\varphi(r(x))\\
&=& \hat\varphi^{-1}\big(((x/P+_P u/P)/2)_{P\in X} +_{G_1} ((x/P+_P u/P)/2)_{P\in X}\big)\\
&=& \hat\varphi^{-1}((x/P+_P u/P)_{P \in X}) = x+u.
\end{eqnarray*}
Therefore, the equation $2r(x)=(x+u)$ has a unique solution $r(x)=(x+u)/2$.
\end{proof}

In addition, if $(G,u)$ is a unital $\ell$-group and for each $x\in M=\Gamma(G,u)$, the element $(x+u)/2$ exists in $M$, then $r(x)=(x+u)/2$, $x\in M$, is a strict square root on $M$.

If $r$ is a strict square root on an MV-algebra $M$, due to the latter theorem, $r(0)=(0+u)/2= 1/2$, and as in the proof of Proposition \ref{pr:tribe}(ii), the MV-algebra of dyadic numbers in the real interval $[0,1]$ can be embedded into $M$ which gives a new proof of \cite[Thm 6.9]{Hol} that was reproved in \cite[Thm 2.4]{Amb}.

Let $M$ be an EMV-algebra with a square root $r$. If $r(0)=0$, $M$ is a generalized Boolean algebra, and $r=\id_M$. If $r$ is strict, $M$ is with top element, and it is equivalent to a strict MV-algebra, so that Theorem \ref{4.3} describes $r$. The following result shows that the latter theorem allows us to describe all square roots on EMV-algebras. It is enough to consider only proper EMV-algebras that are not generalized Boolean algebras.

\begin{thm}\label{th:EMV}
Let $r$ be a square root on a proper EMV-algebra algebra $M$, $r(0)>0$. Then there are a generalized Boolean algebra $M_1$ and a strict EMV-algebra $M_2=\Gamma(G_2,u_2)$ such that $M=M_1\times M_2$, and
\begin{equation}\label{eq:EMV}
r(x_1,x_2)= (x_1,(x_2+u_2)/2)= x_1\vee ((x_2+u_2)/2),\quad x=(x_1,x_2)\in M.
\end{equation}
\end{thm}

\begin{proof}
By Theorem \ref{3.11}, $M$ can be expressed as $M\cong M_1\times M_2$, where $M_1$ is a generalized Boolean algebra and
$M_2=\Gamma(G_2,u_2)$ is a strict EMV-algebra. For simplicity, we assume that $M=M_1\times M_2$. If $r_i=\pi_i\circ r$, $i=1,2$, then $r_1=\id_{M_1}$ and by Theorem \ref{4.3}, $r_2(x_2)=(x_2+u_2)/2$, $x_2\in M_2$ which concludes \eqref{eq:EMV}.
\end{proof}

Finally, we completely characterize square roots on any MV-algebra in a general case, not only on Boolean and strict ones.

\begin{thm}\label{4.3.1}
Let $(M;\oplus,',0,1)$ be an MV-algebra with a square root $r:M\to M$, $(G,u)$ be its corresponding unital $\ell$-group, and set $w=r(0)'\odot r(0)'$. Then, for each $x\in M$, the element $(x\wedge w)\vee((x\wedge w')+w')/2$ exists in $M$, and
\begin{equation}\label{eq:w}
r(x)=(x\wedge w)\vee((x\wedge w')+w')/2,
\end{equation}
where $+$ is the addition in the group $G$.
\end{thm}

\begin{proof}
Set $w:=r(0)'\odot r(0)'$. Due to Proposition \ref{3.2}(xiii), $w$ is a Boolean element of $M$.
If $w=1$, then $r(0)=0$ and $M$ is a Boolean algebra. Moreover, $r=\id_M$ and $(x\wedge w)\vee((x\wedge w')+w')/2=x=r(x)$ for all $x\in M$.

If $w=0$, then $w'=1= r(0)\oplus r(0)$, so that $r$ is strict, and by Corollary \ref{4.2},
$r(x)=(x+1)/2= (x\wedge w)\vee((x\wedge w')+w')/2$.

Let $w\notin\{0,1\}$. By \cite[Thm 2.21]{Hol}, $M\cong M_1\times M_2$, where
$M_1$ is a Boolean algebra $([0,w];\oplus,\lam_w,0,w)$ and $M_2$ is a strict MV-algebra $([0,w'];\oplus,\lam_{w'},0,w')$.
Consider the homomorphisms $f_1:M\to M_1$ and $f_2:M\to M_2$ define by $f_1(x)=x\wedge w$ and $f_2(y)=y\wedge w'$. Clearly,
$f_1$ and $f_2$ are surjective maps. By \cite[Thm 3.1]{Amb},
$t_1(f_1(x))=f_1(r(x))=r(x)\wedge w$ and
$t_2(f(x))=f_2(r(x))=r(x)\wedge w'$ (for all $x\in M$) are square roots on $M_1$ and $M_2$, respectively. Moreover, $r(x)\wedge w= t_1(f_1(x))=f_1(x)=x\wedge w$.
For each $x\in M$, we have
\begin{eqnarray*}
r(x)\wedge w&=& f_1(r(x))=t_1(f_1(x))=f_1(x)=x\wedge w, \mbox{ by Corollary \ref{4.2}}\\
r(x)\wedge w'&=& f_2(r(x))=t_2(f_2(x))=(f_2(x)+w')/2=((x\wedge w')+w')/2, \mbox{ by Theorem \ref{4.3}}.
\end{eqnarray*}
Therefore, for all $x\in M$,
\begin{eqnarray*}
r(x)&=& r(x)\wedge (w\vee w')=(r(x)\wedge w)\vee (r(x)\wedge w')=(x\wedge w)\vee \big((x\wedge w')+w'\big)/2.
\end{eqnarray*}
\end{proof}

Theorem \ref{4.3.1} entails the following characterization of square roots on EMV-algebras.

\begin{cor}\label{co:EMV}
Let $r$ be a square root on an EMV-algebra $M$ and $x$ be an arbitrary element of $M$. Take any idempotent $a\in M$ such that $r(r((0)),x\le a$ and define $w_a= \lambda_a(r(0))\odot \lambda_a(r(0))$. Then
\begin{equation}\label{eq:w(a)}
r(x)=(x\wedge w_a)\vee \big((x\wedge \lambda_a(w_a))+\lambda_a(w_a)\big)/2.
\end{equation}
\end{cor}

\begin{proof}
Let $x\in M$ be given and $a\in \mathcal I(M)$ be such that $r(r((0)),x\le a$. Clearly $r(0)\le a$. Take the MV-algebra $M_a=([0,a];\oplus,\lambda_a,0,a)$ and $M_b=([0,b];\oplus, \lambda_b,0,b)$. Define $r_a= r\wedge a$; it is a square root on $M_a$. Then $r_a(r_a(0))=r_a(r(0)\wedge a)=r_a(r(0))= r(r(0))$.

By Proposition \ref{3.2}(xiii), the element $w_a$ belongs to $\mathcal I(M)$. Since $r(x)=r_a(x)$, using \eqref{eq:w}, we have \eqref{eq:w(a)}.
\end{proof}

\section{Conclusion}

In the paper, we generalized the notion of a square root from MV-algebras to EMV-algebras that are a generalization of MV-algebras where top element is not assumed, but each element is dominated by some idempotent element. On the other hand, every EMV-algebra without top element can be embedded into an EMV-algebra with top element (characterizing EMV-algebra) as its maximal ideal.
Square roots were used to characterize EMV-algebras, e.g. every strict EMV-algebra is with top element, Theorem \ref{3.10}. We found a relation between square roots on EMV-algebras and their representing EMV-algebras, see Theorems \ref{3.6}, Corollary \ref{strict-N}, and Theorem \ref{3.12}. We showed that each EMV-algebra with square root is either a generalized Boolean algebra, a strict EMV-algebra, or is a direct product of a generalized Boolean algebra and a strict EMV-algebra, Theorem \ref{3.11}.

Finally, we presented square roots on tribes, and EMV-tribes. We gave a complete characterization of any square root on an MV-algebra by group addition in the corresponding unital $\ell$-group, see Theorems \ref{4.3}. The application of Theorem \ref{4.3.1} describes all square roots on every EMV-algebra, not only on MV-algebras, Theorem \ref{th:EMV}.

In the future, we hope to study square roots on pseudo MV-algebras.


\end{document}